\documentclass[12pt]{amsart}
\usepackage[pdftex,pdfusetitle,colorlinks,citecolor=blue,urlcolor=blue]{hyperref}
\usepackage{hyperxmp}
\hypersetup{pdfauthor={Alessio D'Al\`i \and Lorenzo Venturello}}
\usepackage{amssymb}
\usepackage{graphicx,color}
\usepackage{amsmath}
\usepackage{comment}
\usepackage{MnSymbol}
\usepackage{cleveref}
\usepackage{enumitem}
\usepackage{algorithm}
\usepackage{algorithmic}
\usepackage{musicography}
\usepackage{tocvsec2}
\newtheorem{thmx}{Theorem}

\numberwithin{equation}{section}

\newtheorem{theorem}{Theorem}[section]
\newtheorem{prop}[theorem]{Proposition}
\newtheorem{lem}[theorem]{Lemma}
\newtheorem{coro}[theorem]{Corollary}
\newtheorem{ques}[theorem]{Question}

\theoremstyle{definition}
\newtheorem{defi}[theorem]{Definition}

\newtheorem{exam}[theorem]{Example}

\newtheorem{rema}[theorem]{Remark}

\newtheorem{nota}[theorem]{Notation}

\newtheorem*{question_intro}{Question}

\DeclareMathOperator{\lk}{\mathrm{lk}}
\DeclareMathOperator{\pd}{\mathrm{pd}}

\DeclareMathOperator{\basis}{\mathbf{1}}

\newcommand{\Ext}{{\text{Ext}}}
\newcommand{\Tor}{{\text{Tor}}}
\newcommand{\Hilb}{{\text{Hilb}}}
\newcommand{\Hom}{{\text{Hom}}}
\newcommand{\cok}{{\text{coker}}}
\newcommand{\init}{{\mathrm{in}}}

\newcommand{\lt}{{\mathrm{in}}}
\newcommand{\reg}{{\text{\normalfont{reg}}}}

\newcommand{\Z}{{\mathbb Z}}
\newcommand{\F}{{\mathbb F}}

\newcommand{\asp}{{\hspace{0pt}a\hspace{-0.3pt}}}
\newcommand{\bsp}{{\hspace{0.28pt}b\hspace{0.3pt}}}

\setlist[enumerate]{label=\roman*.}

\protected\def\ignorethis#1\endignorethis{}
\let\endignorethis\relax
\def\TOCstop{\addtocontents{toc}{\ignorethis}}
\def\TOCstart{\addtocontents{toc}{\endignorethis}}

\makeatletter 
\def\l@subsection{\@tocline{2}{0pt}{2pc}{6pc}{}} 
\makeatother

\title{\texorpdfstring{Koszul Gorenstein algebras\\ from Cohen--Macaulay simplicial complexes}{Koszul Gorenstein algebras from Cohen-Macaulay simplicial complexes}}
\author{Alessio D'Al\`i}
\address{Institut f\"ur Mathematik, Universit\"at Osnabr\"uck, 49069 Osnabr\"uck, Germany (current address) and \newline
Mathematics Institute, University of Warwick, CV4 7AL Coventry, United Kingdom}
\email{alessio.dali@uni-osnabrueck.de}
\author{Lorenzo Venturello}
\address{KTH Royal Institute of Technology, 11428 Stockholm, Sweden}
\email{lven@kth.se ~ lorenzo.venturello@hotmail.it}

\subjclass[2020]{Primary: 05E40; Secondary: 16S37, 13H10, 13D02, 13F65.} 
\keywords{Koszul algebra, quadratic algebra, Gorenstein ring, flag simplicial complex, Cohen--Macaulay simplicial complex, minimal free resolution, Bier sphere, Nagata idealization}

\begin{document}
	\maketitle
	\begin{abstract}
	    We associate with every pure flag simplicial complex $\Delta$ a standard graded Gorenstein $\mathbb{F}$-algebra $R_{\Delta}$ whose homological features are largely dictated by the combinatorics and topology of $\Delta$. 
	    As our main result, we prove that the residue field $\mathbb{F}$ has a $k$-step linear $R_{\Delta}$-resolution if and only if $\Delta$ satisfies Serre's condition $(S_k)$ over $\mathbb{F}$, and that $R_{\Delta}$ is Koszul if and only if $\Delta$ is Cohen--Macaulay over $\mathbb{F}$. Moreover, we show that $R_{\Delta}$ has a quadratic Gr\"{o}bner basis if and only if $\Delta$ is shellable.
	    We give two applications: first, we construct quadratic Gorenstein $\mathbb{F}$-algebras which are Koszul if and only if the characteristic of $\mathbb{F}$ is not in any prescribed set of primes. Finally, we prove that whenever $R_{\Delta}$ is Koszul the coefficients of its $\gamma$-vector alternate in sign, settling in the negative an algebraic generalization of a conjecture by Charney and Davis.
	\end{abstract}

	{
	\hypersetup{linkcolor=black}
	\setcounter{tocdepth}{1}
	\tableofcontents
	
	}
	

	\section{Introduction}
	It is a common line of thought in combinatorial commutative algebra to build correspondences between rings and objects of a more combinatorial nature. This can be insightful for both sides of the story. A prototypical example is Stanley--Reisner theory, where topological and combinatorial properties of simplicial complexes are studied via quotients of polynomial rings by squarefree monomial ideals. In this article we follow the same general principle and associate with any pure simplicial complex $\Delta$ a non-monomial \emph{Gorenstein} standard graded $\mathbb{F}$-algebra $R_{\Delta}$, where $\mathbb{F}$ is a field. If the complex $\Delta$ is flag and satisfies some mild homological condition, the algebra $R_{\Delta}$ will moreover be quadratic.
	
	In combinatorics, Gorenstein algebras show up as Stanley--Reisner rings of homology spheres, as well as Ehrhart rings of certain lattice polytopes. While Gorenstein rings are very rich in structure and well-studied both in algebra and combinatorics, it is not always so easy to construct examples with prescribed features, e.g.~with fixed $h$-vector. A classical and useful tool for this purpose is Nagata idealization, prominently featured in a recent article by Mastroeni, Schenck and Stillman \cite{MSSI}.
	
	A purely combinatorial way to construct a sphere from any simplicial complex $\Delta$ was introduced by Thomas Bier \cite{Bie} (see also \cite{BPSZ, Mur, DFN}): such a \emph{Bier sphere} is PL and simultaneously contains both $\Delta$ and its Alexander dual. In this paper we will take a twist on this construction and crucially exploit a beautiful result by Terai and Yanagawa \cite[Corollary 3.7]{Y} relating Serre's $(S_k)$ property on $\Delta$ to the linearity of the resolution of the Stanley--Reisner ideal of the Alexander dual of $\Delta$.
	
	Our recipe to build the ring $R_{\Delta}$ is as follows: said $I_{\Delta} \subseteq \F[x_1, \ldots, x_n]$ the Stanley--Reisner ideal of a pure simplicial complex $\Delta$ on $n$ vertices, we consider a new complex $\Gamma$ on $2n$ vertices whose Stanley--Reisner ideal is \[I_{\Gamma} := I_{\Delta}+(x_iy_i: i\in [n]) \subseteq \F[x_1, \ldots, x_n, y_1, \ldots, y_n].\] Such a complex turns out to be a PL ball whose boundary is the Bier sphere associated with $\Delta$, as already noted by Murai in \cite{Mur}. Taking inspiration from \cite{MSSI}, we then define the ring $R_{\Delta}$ as the Nagata idealization of the canonical module of $\mathbb{F}[\Gamma]$.
	
	The main contribution of the present paper is to relate the homological features of the links of $\Delta$ with the behaviour of the minimal free resolution of the residue field $\mathbb{F}$ as an $R_{\Delta}$-module, under the extra assumption that $\Delta$ is \emph{flag} (i.e., its Stanley--Reisner ideal is quadratically generated). In order to do this, we build some technical results comparing the resolutions of certain monomial ideals over different rings: see \Cref{cor:same_reg} and the lemmas preceding it. Our first result is as follows:
	
	\begin{thmx}[\Cref{thm: s_k is k-linear}]
	    Let $\Delta$ be a pure flag simplicial complex. The minimal resolution of the residue field $\F$ as an $R_{\Delta}$-module is linear for $k$ steps if and only if $\Delta$ satisfies Serre's condition $(S_k)$ over $\mathbb{F}$.
	\end{thmx}
	 Note that flagness is indeed necessary for the resolution of $\F$ to have a chance to be linear: as the monomials generating $I_{\Delta}$ lie in a minimal generating set of the ideal defining $R_{\Delta}$, the second step of the $R_{\Delta}$-resolution of $\mathbb{F}$ fails to be linear whenever $\Delta$ is not flag.
	 
	Serre's condition $(S_k)$ is an algebraic property that, when formulated for the Stanley--Reisner ring of $\Delta$, translates to a certain vanishing condition about the homology of its links. Reisner's criterion \cite{Rei} is precisely this translation when $k=d$, where $d=\dim \mathbb{F}[\Delta] = \dim(\Delta) + 1$: when $\Delta$ satisfies $(S_d)$, we say that $\Delta$ is Cohen--Macaulay over $\F$.

    The Cohen--Macaulayness of $\Delta$ over $\mathbb{F}$ translates to the $R_{\Delta}$-resolution of $\mathbb{F}$ being linear not just until the $d$-th step, but at \emph{every} step. Standard graded algebras with this behaviour are called \emph{Koszul}. 
	\begin{thmx}[\Cref{cor: koszul_iff_cm}]\label{thm: intro Koszul iff CM}
	Let $\Delta$ be a pure flag simplicial complex.
	    The standard graded Gorenstein $\mathbb{F}$-algebra $R_{\Delta}$ is Koszul if and only if $\Delta$ is Cohen--Macaulay over $\mathbb{F}$.
	\end{thmx}
	
	We remark that this is not the first time the Koszul and Cohen--Macaulay properties have been brought together. Polo \cite{Polo} and Woodcock \cite{Woodcock} independently proved that the incidence algebra associated with a finite graded poset is a (noncommutative) Koszul algebra if and only if all the open intervals of the poset are Cohen--Macaulay, i.e.~the order complex of every open interval is Cohen--Macaulay. In a similar vein, Peeva, Reiner and Sturmfels showed in \cite{PRS} that a graded pointed affine semigroup algebra is Koszul if and only if every open interval of the poset associated with the semigroup is Cohen--Macaulay. These results have later been unified by Reiner and Stamate in \cite{ReSt}. Another intriguing connection between Koszulness and Cohen--Macaulayness has been highlighted by Vallette in the context of operads \cite{Vallette}. Moreover, the study of the Koszul property for quadratic Gorenstein algebras has been the object of intensive research in the last few years \cite{Matsuda, MSSI, MSSII, McSe}. In particular, the results in \cite{MSSI, MSSII, McSe} provide an almost fully complete characterization of the pairs $(c,r)$ for which there exists a non-Koszul quadratic Gorenstein algebra of codimension $c$ and Castelnuovo--Mumford regularity $r$: the only cases that have not been settled yet are $(c,r) = (6,3)$ and $(c,r) = (7,3)$. Unfortunately, the techniques developed in this paper do not give further information on these cases, see \Cref{rem: codim_and_reg}.
		
			We find a pleasing application of \Cref{thm: intro Koszul iff CM} in \Cref{prop: lens spaces}. Fix a finite set $P$ of prime numbers. By considering certain simplicial manifolds $\Delta$ whose vanishing in homology depends on the characteristic of the field, we obtain standard graded quadratic Gorenstein $\F$-algebras $R_{\Delta}$ that are Koszul if and only if the characteristic of $\F$ is not in $P$.
		
	\Cref{thm: intro Koszul iff CM} implies that checking Koszulness for the algebras $R_{\Delta}$ is equivalent to testing if $\Delta$ is Cohen--Macaulay. For instance, it suffices to compute the (finite) minimal free resolution of $\mathbb{F}[\Delta]$ to certify linearity of the (infinite) $R_{\Delta}$-resolution of $\mathbb{F}$. This is in stark contrast with the general behaviour of Koszul algebras. Indeed, working over a field $\F$ of characteristic zero, Roos \cite{Roos} constructed a family $(A_{\alpha})_{\alpha \geq 2}$ of standard graded Artinian quadratic $\F$-algebras with fixed Hilbert series and such that the $(A_{\alpha})$-resolution of $\F$ is linear for $\alpha$ steps, but fails to be so at the $(\alpha+1)$-st. Applying idealization to Roos' construction, McCullough and Seceleanu recently showed in \cite[Theorem 5.1]{McSe} that checking Koszulness is just as hard for quadratic Gorenstein algebras.
	
	Since certifying Koszulness for a quadratic algebra is difficult, algebraists often resort to studying stronger conditions like the existence of a quadratic Gr\"obner basis for the defining ideal.
	Our next result shows that having a quadratic Gr\"{o}bner basis for the algebras $R_{\Delta}$ corresponds indeed to a property of $\Delta$ which is much stronger than Cohen--Macaulayness.
	\begin{thmx}[\Cref{thm: shellable iff qgb}]\label{thm: shellable qgb intro}
	    Let $\Delta$ be a pure flag simplicial complex. The Gorenstein $\mathbb{F}$-algebra $R_{\Delta}$ has a quadratic Gr\"{o}bner basis if and only if $\Delta$ is a shellable complex.
	\end{thmx}
	As the ideal defining $R_{\Delta}$ has a characteristic-free generating set consisting of monomials and binomials of the form $\mathbf{m} - \mathbf{m'}$ (\Cref{prop:RDelta_presentation}), the existence of a quadratic Gr\"{o}bner basis for such an object is a combinatorial property not depending on the choice of the field. The same is true for shellability of $\Delta$.
	
	As a last topic, we investigate a numerical invariant related to the Hilbert series of $R_{\Delta}$. Gorenstein algebras are known to have a symmetric $h$-polynomial with nonnegative coefficients. Gal \cite{Gal} proposed the study of a certain linear transformation of these coefficients, called the $\gamma$-vector, and conjectured the nonnegativity of its entries in the case of Stanley--Reisner rings of flag $\mathbb{F}$-homology spheres.
	The validity of one of the inequalities given by Gal's conjecture is known as Charney--Davis conjecture \cite{CD} and has important implications in metric geometry: see \cite{Fo} for an excellent survey on the topic.
	
	The Stanley--Reisner ring of a flag $\F$-homology sphere happens to be both Koszul and Gorenstein: using this observation as a starting point, Reiner and Welker began a more algebraic investigation of the Charney--Davis conjecture in \cite{ReWe}. An explicit generalization of the Charney--Davis conjecture in this direction can be found in a survey by Peeva and Stillman:
		\begin{question_intro}[{\cite[Problem 10.3]{PeSt}}] \label{qu:PeSt}
	    Let $S/I$ be a Koszul Gorenstein algebra with $h$-vector $(h_0, h_1, \ldots, h_{2e})$. Is it true that $(-1)^e(h_0 - h_1 + h_2 - \ldots + h_{2e}) \geq 0$?
	\end{question_intro}
	
	 The answer turns out to be negative. This is a consequence of our last result:
	
	\begin{thmx}[\Cref{prop: formula gamma}] \label{thm: intro gamma}
	    Let $\Delta$ be a pure $(d-1)$-dimensional simplicial complex. The vector $\gamma(R_{\Delta})$ is given by $\gamma_0(R_{\Delta}) = 1$, 
	    $\gamma_1(R_{\Delta}) = 2h_1(\Delta) + \sum_{k=2}^{d}h_k(\Delta)$, and
	    \[
	    \gamma_i(R_{\Delta})=(-1)^{i-1}\sum_{k=2i-1}^d \left(\binom{k-i}{i-1} + \binom{k-i-1}{i-2}\right)h_k(\Delta)
	    \]
	    for $2 \leq i \leq \lfloor\frac{d+1}{2}\rfloor$. In particular, if $h(\Delta)$ has nonnegative entries, then $(-1)^{i-1}\gamma_i(R_{\Delta})\geq 0$.
	\end{thmx}
	As the $h$-vector of a Cohen--Macaulay complex has nonnegative entries, \Cref{thm: intro gamma} and \Cref{thm: intro Koszul iff CM} show that whenever the algebra $R_{\Delta}$ is Koszul, its $\gamma$-vector has entries which alternate in sign. In particular, when $\Delta$ is a $(d-1)$-dimensional Cohen--Macaulay flag complex with $d\equiv 3 \text{ mod }4$, the Charney--Davis quantity $\gamma_{\lfloor \frac{d+1}{2}\rfloor}(R_{\Delta})$ is nonpositive. For an explicit example where $\gamma_{\lfloor \frac{d+1}{2}\rfloor}(R_{\Delta}) < 0$, see \Cref{ex: algebraic CD counterexample}.\\
	Finally, in \Cref{sec: GZ} we highlight a connection to the literature, showing that a Gorenstein algebra defined via apolarity by Gondim and Zappal\`{a} \cite{GZ} is an Artinian reduction of $R_{\Delta}$ in characteristic zero. As a consequence, \Cref{thm: intro Koszul iff CM} applies to the algebras in \cite{GZ} as well. Moreover, we show that \Cref{thm: shellable qgb intro} also goes through the Artinian reduction, and thus provides a criterion to decide if the algebras studied by Gondim and Zappal\`a have a quadratic Gr\"{o}bner basis. Finally, in \Cref{rem: quadratic ADelta} we fix a previous characterization \cite[Theorem 3.5]{GZ} of quadraticity for such algebras.
	
	\TOCstop

	\TOCstart
	
	\section{Preliminaries}
	\subsection{Graded algebras and modules}
	Let us fix a field $\F$ and consider a standard graded $\F$-algebra $R=\F[x_1,\dots,x_n]/I$, for some homogeneous ideal $I$. The \emph{Hilbert series} of $R$ is a rational function of the form \[\Hilb(R,t)=\frac{\sum_{i=0}^s h_i t^i}{(1-t)^d},\]
    with $\sum_i h_i\neq 0$. The numerator of $\Hilb(R,t)$ expressed as above is the \emph{$h$-polynomial} of $R$, and the integer sequence $h(R)=(h_0,\dots,h_s)$ is the \emph{$h$-vector} of $R$. In general, the entries of the $h$-vector can be either positive or negative.
    
    If $M$ is a finitely generated graded module over $R$, we can consider its \emph{minimal graded free resolution}. This is the unique (up to isomorphism of chain complexes) complex of free $R$-modules and degree zero maps
    \begin{equation}\label{eq: resolution}
         \cdots\overset{\partial_{i+1}}{\to} \bigoplus_j R(-j)^{\beta_{i,j}}\overset{\partial_{i}}{\to} \cdots \overset{\partial_{2}}{\to} \bigoplus_j R(-j)^{\beta_{1,j}} \overset{\partial_{1}}{\to} \bigoplus_j R(-j)^{\beta_{0,j}}\to 0
    \end{equation}    
    which is exact in all positions but the zeroth, where $\cok(\partial_1)\cong M$, and is such that every $\partial_i$ can be represented as a matrix of zeros and homogeneous polynomials of positive degree. The numbers $\beta_{i,j}$ in \eqref{eq: resolution} are known as the \emph{graded Betti numbers} of $M$. In a more functorial fashion, they can be written as $\beta_{i,j}=\beta^R_{i,j}(M)=\dim_{\F}\Tor_i^R(M,\F)_j$. This highlights a second way to compute graded Betti numbers: namely, thanks to the commutativity of $\Tor$, one can tensor a minimal resolution of the $R$-module $\F$ with $M$, and then take the homology of the resulting chain complex. We will use this standard observation in the next sections. The quantity $\sup\{j - i : \beta_{i,j}^R(M) \neq 0\}$ is known as the \emph{Castelnuovo--Mumford regularity} of $M$, denoted by $\mathrm{reg}_R(M)$.
    
    The standard tool to record the information about graded Betti numbers is the \emph{Poincar\'{e} series} of $M$ over $R$, defined as
\begin{equation} \label{eq: def_poincare}
    P_R^M(s,t)=\sum_{i,j}\beta^R_{i,j}(M)s^j t^i.
\end{equation}
We will omit the superscript when $M$ is the residue field $\F = R/(x_1, \ldots, x_n)$.

    If $M$ is generated in a single degree, we say that the minimal resolution of $M$ is \emph{linear for $0$ steps}; if moreover all the nonzero entries of the matrices $\partial_i$ in \eqref{eq: resolution} are linear forms for every $1 \leq i \leq k$, we say that the minimal resolution of $M$ is \emph{linear for $k$ steps}. If the nonzero entries of \emph{all} matrices in the resolution are linear forms, we say that $M$ has a \emph{linear resolution}. 
    
    By a celebrated theorem of Hilbert, minimal resolutions of modules over the polynomial ring $S=\F[x_1, \ldots, x_n]$ are finite: the length of the minimal free resolution of $M$ as an $S$-module is known as the \emph{projective dimension} of $M$ and denoted by $\pd_S(M)$. When $M = R$, it is known that $n-d \leq \pd_S(R) \leq n$, where $d$ is the Krull dimension of $R$. Algebras attaining the above lower bound are well studied:

    \begin{defi}
        A $d$-dimensional standard graded $\F$-algebra $R=S/I$ is \emph{Cohen--Macaulay} over $\F$ if $\pd_S(R)=n-d$.
    \end{defi}
    Cohen--Macaulay algebras play an important role in commutative algebra, algebraic geometry and combinatorics \cite{BH,St_green}, and so does the special subclass consisting of \emph{Gorenstein} algebras. 
    \begin{defi}
        A $d$-dimensional standard graded $\F$-algebra $R=S/I$ is \emph{Gorenstein} over $\F$ if it is Cohen--Macaulay and $\dim_{\F} \Tor^S_{n-d}(R,\F)=1$.
    \end{defi}
    We conclude this section by introducing an object that will play a crucial role in the rest of the paper.
    
    \begin{defi} \label{def: canonical}
        Let $S=\F[x_1, \ldots, x_n]$ be a standard graded polynomial ring and let $R=S/I$ be a $d$-dimensional standard graded Cohen--Macaulay $\F$-algebra. Then:
        \begin{itemize}
            \item the \emph{canonical module} of $S$ is the $\mathbb{Z}^n$-graded module $\omega_S := S(-1, -1, \ldots, -1)$ (hence, if we are only interested in the $\mathbb{Z}$-graded structure, $\omega_S \cong S(-n)$);
            \item the canonical module of $R$ is the $\mathbb{Z}$-graded module $\omega_R := \Ext^{n-d}_S(R, \omega_S)$. If $I$ happens to be $\mathbb{Z}^n$-graded, then so is $\omega_R$; 
            \item the \emph{$a$-invariant} of $R$ is $a(R) := -\min\{j \in \mathbb{Z} : (\omega_R)_j\neq 0\}$;
            \item the algebra $R$ is \emph{level} if $\omega_R$ is generated in a single $\mathbb{Z}$-degree.
        \end{itemize}
    \end{defi}
    
	The minimal free resolution of $\omega_R$ as an $S$-module can be obtained dualizing a minimal free resolution of $R$ via the contravariant functor $\Hom_S(-,\omega_S)$.
    
	\subsection{Simplicial complexes}
	Let $\Delta$ be a simplicial complex on $[n]=\{1,\dots,n\}$, i.e., a collection of subsets of $[n]$ closed under inclusion. Elements of $\Delta$ are called \emph{faces} and inclusion-maximal faces are called \emph{facets}. We denote by $\mathcal{F}(\Delta)$ the set all facets of $\Delta$. The dimension af a face equals its cardinality minus one, and the dimension of $\Delta$ is the maximal dimension of one of its faces; if all facets have the same dimension, we say that $\Delta$ is \emph{pure}. We record the number of faces in each dimension in a vector $f(\Delta)$ called the \emph{$f$-vector} of $\Delta$. If $\Delta$ is $(d-1)$-dimensional, then $f(\Delta)=(f_{-1},f_0,\dots,f_{d-1})$, where $f_{-1}=1$, unless $\Delta=\emptyset$.
	
	If $F_1, \ldots, F_r$ are subsets of $[n]$, we will denote by $\langle F_1, \ldots, F_r \rangle$ the simplicial complex generated by $F_1, \ldots, F_r$, i.e.~the smallest simplicial complex on $[n]$ containing $F_1, \ldots, F_r$.
	
	Fix now a field $\F$. We denote by $I_{\Delta}\subseteq \F[x_1,\dots,x_n]$ the squarefree monomial ideal generated by the monomials supported on the complement of $\Delta$ in $2^{[n]}$, namely 
	\[
	    I_{\Delta}:=\left(\mathbf{x}^F ~ : ~ F\notin \Delta\right),
	\]
	where $\mathbf{x}^F=\prod_{i\in F}x_i$.
	
	The ideal $I_{\Delta}$ is known as the \emph{Stanley--Reisner ideal} of $\Delta$, and the graded $\F$-algebra $\F[\Delta]=\F[\mathbf{x}]/I_{\Delta}$ is the associated \emph{Stanley--Reisner ring}. There is a rich interplay between combinatorial properties of $\Delta$ and algebraic features of $\F[\Delta]$: for instance, the Krull dimension of $\F[\Delta]$ is the dimension of $\Delta$ plus one. We can actually do better than this, as we can read off the whole $f$-vector of $\Delta$ from the Hilbert series of $\F[\Delta]$ and vice versa: the $h$-vector $h(\Delta)=h(\F[\Delta])=(h_0,\dots,h_d)$ is an invertible integer linear transformation of $f(\Delta)$. More precisely: 
	\begin{equation}\label{eq: h from f}
	    h_i=\sum_{j=0}^{i}(-1)^{i-j}\binom{d-j}{d-i}f_{j-1}
	    \end{equation}
	\begin{equation} \label{eq: f from h}
    f_{i-1}=\sum_{j=0}^{i}\binom{d-j}{d-i}h_j.
	\end{equation}

	Local properties of a simplicial complex around a face $F\in\Delta$ are described by the \emph{link} of $F$, defined as
	\[
	    \lk_{\Delta}(F)~:=~\{G\in \Delta ~:~ F\cap G=\emptyset, F\cup G \in\Delta\}.
	\]
	
	If $\Delta$ is a pure $(d-1)$-dimensional complex, then $\lk_{\Delta}(F)$ is a (pure) $(d-|F|-1)$-dimensional simplicial complex. As a slogan, the simpler the homology of the links of $\Delta$ is, the nicer $\Delta$ is. To make this precise, we give the following definition after Murai and Terai \cite{MT} (even though it was already implicit in work by Terai and Yanagawa \cite{Y}):
	
	\begin{defi}\label{def: Serre} Let $r \geq 1$. A simplicial complex $\Delta$ satisfies the (combinatorial) \emph{Serre condition} $(S_r)$ with respect to a field $\F$ if
	\[
	    \widetilde{H}_{i}(\lk_{\Delta}(F);\F) = 0
	\]
	for every $F\in\Delta$ and for every $i<\min\{r-1,\dim(\lk_{\Delta}(F))\}$.
	A $(d-1)$-dimensional simplicial complex satisfying property $(S_{d})$ over $\F$ is called \emph{Cohen--Macaulay} over $\F$.  
	\end{defi}
	
	\Cref{def: Serre} deserves some comments. It is clear that every complex $\Delta$ must satisfy $(S_1)$ and that, if $\Delta$ has $(S_r)$, then it has $(S_i)$ for every $1\leq i\leq r$. Moreover, $\Delta$ satisfies property $(S_2)$ if and only if it is pure and the link of any face of codimension at least two is connected. In particular, property $(S_2)$ does not depend on the field. For $r \geq 2$, condition $(S_r)$ states that the link of any face of codimension at most $r$ is allowed to have nonvanishing homology only in top homological degree, while lower-dimensional faces need to have vanishing homology up to homological degree $r-2$.
	The combinatorial Serre condition $(S_r)$ for $\Delta$ (with respect to $\F$) is equivalent to the usual algebraic Serre condition $(S_r)$ for the Stanley--Reisner ring $\F[\Delta]$. In particular, a $(d-1)$-dimensional complex $\Delta$ is $(S_d)$ if and only if $\F[\Delta]$ satisfies the algebraic $(S_d)$ condition, which in turn means that $\F[\Delta]$ is a Cohen--Macaulay ring. This fact is known as \emph{Reisner's criterion} \cite{Rei}. The equivalence between the algebraic and combinatorial Serre conditions for simplicial complexes is known and proved for instance in \cite[end of Section 1]{Terai}, but we will sketch a proof here for the interested reader.
	
	\begin{prop}
	    Let $r \geq 1$ and let $\Delta$ be a simplicial complex on $[n]$. Then $\Delta$ satisfies the combinatorial $(S_r)$ condition (with respect to $\F$) if and only if the Stanley--Reisner ring $\F[\Delta]$ satisfies the usual algebraic $(S_r)$ condition, i.e.
	    \begin{equation} \label{eq: algebraic_serre}
	        \mathrm{depth}\!\ \F[\Delta]_{\mathfrak{p}} \geq \min\{r, \ \dim \F[\Delta]_{\mathfrak{p}}\} \hspace{30pt} \textrm{for every }\mathfrak{p}\in \mathrm{Spec}(\F[\Delta]).
	   \end{equation}
	\end{prop}
	\begin{proof}
	    For a Noetherian ring $R$, the algebraic $(S_1)$ condition is known to be equivalent to the absence of embedded primes in $R$, which is always the case for the reduced ring $\F[\Delta]$. On the other hand, the combinatorial $(S_1)$ condition is met by every simplicial complex.
	    
	    Let now $r \geq 2$, and denote by $S$ the polynomial ring $\F[x_1, \ldots, x_n]$. Note that both the combinatorial and the algebraic $(S_2)$ conditions imply that $\F[\Delta]$ is equidimensional, i.e. $\Delta$ is pure: for the algebraic statement, see \cite[Remark 2.4.1]{Hart}. By \cite[Proposition 2.11]{DMV}, \eqref{eq: algebraic_serre} is then equivalent to
	    \begin{equation*}
	        \dim\Ext^{n-j}_S(\F[\Delta], \omega_S) \leq j-r \hspace{30pt} \text{for every } j<\dim \F[\Delta]
	    \end{equation*}
	    which, since $\Ext^{n-j}_S(\F[\Delta], \omega_S)$ is a squarefree module \cite{Y}, is in turn equivalent to
	    \begin{equation*} \label{eq: squarefree_ext}
	        \Ext^{n-j}_S(\F[\Delta], \omega_S)_F = 0 \hspace{20pt} \text{for every } j<\dim(\Delta) + 1 \text{ and } F \in \{0,1\}^n \text{ with } |F| > j-r.
	    \end{equation*}
	    It is a useful fact of Stanley--Reisner theory that $\Ext^{n-j}_S(\F[\Delta], \omega_S)_F \cong H_{j - |F| - 1}(\lk_{\Delta}F, \F)$: see for instance \cite[Proposition 3.1]{Y}. To conclude it is then enough to set $i = j - |F| - 1$ and note that, since $\Delta$ is pure, $\dim \lk_{\Delta}F = \dim{\Delta}-|F|$.
	\end{proof}
	
	The \emph{Alexander dual} of ${\Delta}$ is the simplicial complex $\Delta^*$ on the same vertex set as $\Delta$ and with Stanley--Reisner ideal \[I_{\Delta^*}= (\mathbf{x}^{[n]\setminus G} : G\in\Delta),
	\]
	where $\mathbf{x}^{[n]\setminus G} = \prod_{i \in [n] \setminus G}x_i$.
	A minimal generating set for $I_{\Delta^*}$ is hence given by monomials supported on the complements of facets of $\Delta$ in $[n]$.
	
	Eagon and Reiner \cite{ER} related the Cohen--Macaulay property of $\Delta$ to the Castelnuovo--Mumford regularity of the Alexander dual ideal $I_{\Delta^*}$. This result was extended to $(S_r)$ conditions by Terai and Yanagawa \cite{Y}.
	
	\begin{theorem} \label{thm:ERTY}
	    \cite[Theorem 3]{ER}, \cite[Corollary 3.7]{Y}
	    Let $\Delta$ be a $(d-1)$-dimensional simplicial complex on $[n]$. For every $2\leq r\leq d$, the complex $\Delta$ satisfies property $(S_r)$ if and only if $I_{\Delta^*}$ has a resolution as an $\F[x_1,\dots,x_n]$-module which is linear for $r-1$ steps. In particular, $\Delta$ is Cohen--Macaulay if and only if $I_{\Delta^*}$ has a linear $\F[x_1, \ldots, x_n]$-resolution.
	\end{theorem}
	
	\subsection{Nagata idealization}
	
	A theme of this paper will be the construction of Gorenstein algebras from ``good enough'' objects. The required technical tool is a well-known operation called \emph{idealization}, made popular by Nagata. This subsection follows closely \cite[Section 3]{MSSI}, and we refer the reader to that for more information.
	
	\begin{defi}
		The \emph{idealization} of an $A$-module $M$, denoted $A \ltimes M$, is the $A$-algebra with underlying module $A \oplus M$ and multiplication defined by $(a_1,m_1)\cdot (a_2,m_2)=(a_1a_2, a_1m_2+a_2m_1)$.
	\end{defi}
	This operation is called idealization because it turns every submodule $N$ of $M$ into an ideal $\{0\}\ltimes N$ of $A\ltimes M$. For the rest of the paper, the ring $A$ will be a standard graded $\F$-algebra. If $M$ is graded, then so is $A \ltimes M$, by setting $(A \ltimes M)_j = A_j \oplus M_j$; moreover, $A \ltimes M$ is standard graded if and only if $M$ is generated in degree one \cite[Remark 3.1]{MSSI}. Note that, if $A$ and $M$ are $\mathbb{Z}^n$-graded, then $A \ltimes M$ inherits the $\mathbb{Z}^n$-grading.

We record here for later use a lemma that relates the homological information of an idealization to that of its building blocks. 

\begin{lem}\label{lem: Poincare series}\cite[Theorem 2]{G}
   Let $A$ be a standard graded $\F$-algebra and let $M$ be a finitely generated graded $A$-module generated in degree one. Then 
        \begin{equation} \label{eq:poincare}
        P_{A \ltimes M}(s,t) = \dfrac{P_{A}(s,t)}{1-t P^{M}_{A}(s,t)},
        \end{equation}
        where $P^{M}_{A}(s,t)$ is the Poincar\'e series of $M$ as an $A$-module, as defined in \eqref{eq: def_poincare}.
\end{lem}

	\begin{defi}
		A standard graded Cohen--Macaulay $\F$-algebra $A$ is \emph{superlevel} if it is level and its canonical module $\omega_A$ has a linear presentation as an $A$-module, i.e., there is an exact sequence
		\[
			A(a-1)^{b}~\overset{\varphi_1}{\rightarrow}~ A(a)^{g}~\overset{\varphi_0}{\rightarrow}~ \omega_A~\to ~0,
		\]
		with $a = a(A)$ as in \Cref{def: canonical}.
	\end{defi}
	
	We end this subsection by collecting several results proved in \cite[Section 3]{MSSI}. Parts iv and v follow immediately by analyzing the graded structure of the idealization and using the formula for the Hilbert series of the canonical module \cite[Corollary 4.4.6]{BH}.
	
	\begin{theorem}[{\cite[Proposition 3.2, Lemma 3.3]{MSSI}}] \label{thm:idealization}
		Let $A = S/I$ be a standard graded and level $\F$-algebra and let $\widetilde{A} := A \ltimes \omega_A(-a(A)-1)$. The following statements hold:
		\begin{enumerate}
			\item $\widetilde{A}$ is a standard graded and Gorenstein $\F$-algebra;
			\item if $\omega_A$ is minimally generated by $g$ elements, then
		\begin{equation} \label{eq: idealization_presentation}
			\widetilde{A}~=~ \frac{S[z_1,\dots z_g]}{I + \mathcal{L} + (z_1, \dots, z_g)^2}, 
		\end{equation}
		where 
		\begin{equation}
			\mathcal{L}~:=~\left(\sum_{i=1}^g f_iz_i ~ : ~ (f_1,\dots, f_g)\in \mathrm{Syz}^S_1(\omega_A)\right);
		\end{equation}
			\item if $A$ is quadratic, then $\widetilde{A}$ is quadratic if and only if $A$ is superlevel.
			\item $h_i(\widetilde{A})=h_i(A)+h_{s-i+1}(A)$ for every $1<i<s-1$. 
			\item $\dim(\widetilde{A}) = \dim(A)$.
		\end{enumerate}

	\end{theorem}
	
		\subsection{Koszul algebras}
The last ingredient we need is a special class of standard graded algebras called \emph{Koszul algebras}: for a survey, we direct the reader to \cite{CDR}. 

\begin{defi}
    A standard graded $\F$-algebra $A$ is \emph{Koszul} (over $\F$) if $\F$ has a linear resolution as an $A$-module.
\end{defi}
In the above definition, $\F$ is identified with the quotient of $A$ by its maximal homogeneous ideal. Observe that, by definition, $A$ is Koszul if and only if $P_{A}(s,t)\in\Z[[st]]$.

Due to the fact that a minimal resolution of $\F$ as an $A$-module is typically infinite, it is in general very hard to prove that a certain algebra $A$ is Koszul. However, the following is well-known.

\begin{prop}
    Let $A=\F[x_1, \ldots, x_n]/I$ be a standard graded $\F$-algebra and assume without loss of generality that $I \subseteq (x_1, \ldots, x_n)^2$. Then:
    \begin{itemize}
    \item if $A$ is Koszul, then $I$ is generated in degree $2$;
    \item if $I$ has a Gr\"{o}bner basis of quadrics, then $A$ is Koszul.
    \end{itemize}
    
\end{prop}
Having a Gr\"{o}bner basis of quadrics is actually a much stronger condition than Koszulness, and we will see in \Cref{sec: shellable} how this translates combinatorially for the algebras studied in the present paper.

We close this section by a statement highlighting how Koszulness interacts with suitable idealizations.

\begin{lem}\label{lem: from omega to F}
        Let $A$ be a Koszul $\F$-algebra and let $M$ be a finitely generated graded $A$-module generated in degree one. Then the residue field $\F$ has an $(A \ltimes M)$-resolution which is linear for $k$ steps if and only if the $A$-module $M$ has a resolution which is linear for $k-1$ steps. In particular, $A \ltimes M$ is Koszul if and only if $M$ has a linear $A$-resolution.
\end{lem}
\begin{proof}
    The statement follows from comparing the coefficients in \eqref{eq:poincare}. Since $A$ is Koszul, one has that $P_{A}(s,t)\in \Z[[st]]$. Let $P_{A \ltimes M}(s,t)=\sum_{i,j} b_{i,j} s^j t^i$ and $P^{M}_{A}(s,t)=\sum_{i,j} c_{i,j} s^j t^i$. After multiplying both sides of \eqref{eq:poincare} by $1-tP^{M}_{A}(s,t)$, we analyze the coefficient of $s^jt^i$ for every $i, j$ with $i < j$, obtaining that
    \begin{equation} \label{eq:b_and_c}
        b_{i,j}-\sum_{h,\ell}b_{h,\ell} c_{i-h-1,j-\ell}=0.
    \end{equation}
    Note that $b_{0,0}=1$ and $b_{0,j}=0$ for every $j>0$. We want to show that
    \begin{equation*}
    b_{i,j} = 0\text{ for every }1 \leq i \leq k, \ j > i \Leftrightarrow c_{i-1,j} = 0 \text{ for every }1 \leq i \leq k, \ j > i.
    \end{equation*}
    
    To prove the left to right arrow it is enough to observe that, for every $1 \leq i \leq k$ and $j > i$, \eqref{eq:b_and_c} implies that $0 = b_{i,j} \geq b_{0,0}c_{i-1,j} = c_{i-1,j} \geq 0$.
    
    Conversely, let us now assume that $c_{i-1,j} = 0$ for every $1 \leq i \leq k$ and $j > i$, and assume by contradiction that there exists $1 \leq m \leq k$ such that $b_{m,q}>0$ for some $q > m$. We can then pick this $m$ to be the smallest possible with respect to this property. Because of \eqref{eq:b_and_c}, it holds that $b_{m,q} = \sum_{h,\ell}b_{h, \ell}c_{m-h-1,q-\ell}$. Since every $h$ in the sum must be strictly smaller than $m$, it follows that $b_{h,\ell}$ vanishes whenever $h < \ell$. Hence, $0 < b_{m, q} = \sum_h b_{h,h}c_{m-h-1, q-h}$, and hence there must be at least one nonvanishing $c_{m-h-1, q-h}$, contradicting the hypothesis.  
\end{proof}

	\settocdepth{section}
	
	\section{Bier balls and spheres}
	
	In this section we recall a simple construction which associates with any simplicial complex on $[n]$ (other than the full simplex) an $(n-1)$-dimensional PL ball which is a proper subcomplex of the boundary complex of the $n$-dimensional cross-polytope. Such complexes were notably studied in higher generality by Murai \cite{Mur}.

	\begin{defi}\label{def: whiskers}
		Let $\Delta$ be a simplicial complex on $n$ vertices, and let $I_{\Delta}\subseteq \F[\mathbf{x}]$ be its Stanley--Reisner ideal. We define a new complex $\Gamma$ on $2n$ vertices by
		\[
			I_{\Gamma}~:=~I_{\Delta}~+~( x_iy_i : i\in [n] ) ~ \subseteq ~ \F[\mathbf{x},\mathbf{y}].
		\]
		If $\Delta$ is not the $(n-1)$-dimensional simplex, we will call $\Gamma$ the \emph{Bier ball} associated with $\Delta$: see \Cref{prop: Bier dual} below.
	\end{defi}
	\begin{rema}
	     Most of the extant literature does not really deal with Bier balls, but rather with the PL spheres that bound them: these are known as \emph{Bier spheres} after Thomas Bier, who introduced them in \cite{Bie} (see also \cite{BPSZ,DFN}).
	\end{rema}
	\begin{rema}
	    For $\Delta$ flag, the construction in \Cref{def: whiskers} has already appeared in the literature several times, notably in \cite[Proposition 4.1]{CVflag}. The earliest reference we are aware of is \cite[Proposition 2.2]{Vil}, which has a slightly different take on the subject. Indeed, if $\Delta$ is flag, the ideal $I_{\Delta}$ can be seen as the edge ideal of a graph $G$, and $I_{\Gamma}$ becomes the edge ideal of the graph obtained from $G$ by adding extra leaves (or ``whiskers'') to every vertex.
	\end{rema}
	
	\begin{prop}\label{prop: Bier dual}
			Let $\Delta$ be a simplicial complex on $[n]$ and let $\Gamma$ be as in \Cref{def: whiskers}. Then:
			\begin{enumerate}
                \item $\Gamma$ is a shellable $(n-1)$-dimensional simplicial complex;
                \item $h(\Gamma)=f(\Delta)$;
                \item if $\Delta$ is not the $(n-1)$-simplex, then $\Gamma$ is a PL $(n-1)$-ball and its boundary complex $\partial\Gamma$ is the PL $(n-2)$-sphere known as the \emph{Bier sphere} associated with $\Delta$ \cite{Bie};
			    \item the canonical module $\omega_{\F[\Gamma]}$ is isomorphic as an $\F[\Gamma]$-module to the ideal of $\F[\Gamma]$ generated by the monomials $\mathbf{y}^{[n] \setminus F}$, where $F$ ranges over all facets of $\Delta$ (here $\mathbf{y}^{\varnothing}$ is equal to $1$). In other words, the canonical module of $\F[\Gamma]$ is the image of the Alexander dual ideal of $\Delta$ (in the $y$-variables) under the projection $\F[\mathbf{x}, \mathbf{y}] \twoheadrightarrow \F[\Gamma]$.
			\end{enumerate}
	\end{prop}
	We note that \Cref{prop: Bier dual} was essentially already proved in \cite[Lemma 1.4, Theorem 1.14 and Theorem 3.6]{Mur}, where a more general construction associating a Bier ball with every multicomplex is studied. However, we will sketch our approach below for the reader's benefit.
	
    First of all, note that the facets of $\Gamma$ are precisely the sets $F^{\sharp} = \{x_i \mid i \in F\} \cup \{y_j \mid j \notin F\}$, where $F$ ranges over all faces of $\Delta$. Order (partially) the facets of $\Gamma$ so that $F^{\sharp} \prec G^{\sharp}$ when $\dim(F) < \dim(G)$. It is left to the reader to show that any total order refining $\prec$ gives a shelling order for $\Gamma$ (for the definition of a shelling order, see \Cref{def: shelling}). Moreover, since the minimal new face we are adding at each step of the shelling has cardinality $|F|$, the claim about the $h$-vector of $\F[\Gamma]$ follows from \cite[Proposition III.2.3]{St_green}. 
    
	Now note that every face of $\Gamma$ of codimension $1$ is contained in at most two facets. Indeed, every codimension $1$ face is of the form $F^{\sharp} \setminus\{v_i\}$, where $i \in [n]$ and $v_i$ is either $x_i$ or $y_i$. As $\{x_j,y_j\}$ is a nonface of $\Gamma$ for every $j \in [n]$, one can extend the given codimension $1$ face to a facet only by adding either $y_i$ (which always yields a facet) or $x_i$ (which might yield a facet, depending on $\Delta$). By \cite[Theorem 11.4]{BjTOP}, the shellable complex $\Gamma$ must then be either a PL ball or a PL sphere, with the former case occurring when there exists a codimension $1$ face contained in exactly one facet. This happens whenever $\Delta$ has at least a missing face, i.e., it is not the full $(n-1)$-simplex.
	
	Finally, the last statement follows from the description of the canonical module of any homology ball (see for instance \cite[Theorem 5.7.1]{BH}), which says that $\omega_{\F[\Gamma]}$ is the image of $I_{\partial \Gamma}$ under the projection $\mathbb{F}[\mathbf{x},\mathbf{y}]\twoheadrightarrow \mathbb{F}[\Gamma]$. Recalling that $\partial\Delta$ consists of the codimension $1$ faces of $\Gamma$ contained in exactly one facet, one checks that $I_{\partial \Gamma} = I_{\Gamma} + (\mathbf{y}^{[n] \setminus F} \mid F \in \mathcal{F}(\Delta))$. This is precisely the definition of a Bier sphere, see for instance \cite[Corollary 5.3]{DFN}. Note that if $\Delta$ is the $(n-1)$-simplex, then $\Gamma$ is actually a sphere, $\F[\Gamma]$ is Gorenstein and $\omega_{\F[\Gamma]} \cong \F[\Gamma]$; since $\mathbf{y}^{\varnothing} = 1$, we are done also in this case.
	
	The simplicial complex $\Gamma$ is actually vertex decomposable \cite[Remark 1.8]{Mur}, a property which is stronger than shellability.  Moreover, any Bier $(n-1)$-ball or $(n-2)$-sphere is naturally a subcomplex of the boundary complex $\Diamond_n$ of the $n$-cross-polytope, with $I_{\Diamond_n} = (x_iy_i : i\in [n])$.
	
	\begin{exam}\label{ex: Bier}
		Let $n=3$ and $\Delta$ be the simplicial complex $\{\varnothing, x_1,x_2,x_3,x_1x_2\}$. Then $I_{\Gamma}= (x_1x_3, x_2x_3) + (x_1y_1,x_2y_2,x_3y_3)$, and $\Gamma$ is the $2$-dimensional pure shellable homology ball with shelling order \[y_1y_2y_3 \prec \underline{x_1}y_2y_3 \prec y_1\underline{x_2}y_3 \prec y_1y_2\underline{x_3} \prec \underline{x_1x_2}y_3.\] Observe that $h(\Gamma)=f(\Delta)=(1,3,1)$. The complex $\partial\Gamma = \langle x_1x_2, x_2y_1,y_1x_3,y_2x_3,x_1y_2\rangle$ is a $5$-cycle.
	\end{exam}
	\begin{figure}[h]
	    \centering
	    \includegraphics{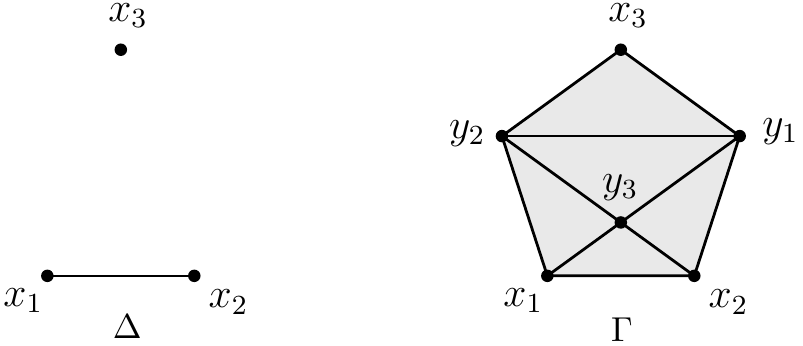}
	    \caption{The simplicial complex $\Delta$ as in \Cref{ex: Bier} and the associated Bier ball $\Gamma$.}
	    \label{fig:my_label}
	\end{figure}
	
	\begin{rema}
		The Bier sphere  $\partial\Gamma$ in \Cref{ex: Bier} is flag, but this is usually not the case in general. In \cite{HK}, the authors prove that there are only $8$ possible complexes $\Delta$ which give rise to flag Bier spheres.
	\end{rema}

	\section{\texorpdfstring{From Bier balls to Gorenstein algebras: the ring $R_{\Delta}$}{From Bier balls to Gorenstein algebras: the ring RDelta}}
	For the rest of this article we will be concerned with idealizing the canonical module of the Stanley--Reisner ring of a Bier ball. We set some notation for the rest of the paper.
	
	\begin{nota} \label{notat:main}
	From now on, if not specified otherwise:
	\begin{itemize}
	    \item $\Delta$ is a pure $(d-1)$-dimensional simplicial complex on $n$ vertices, and its Stanley--Reisner ideal $I_{\Delta}$ lives in $\F[x_1, \ldots, x_n]$;
	    \item $S$ is the polynomial ring $\F[\mathbf{x}, \mathbf{y}]=\F[x_1, \ldots, x_n, y_1, \ldots, y_n]$;
	    \item $\Gamma$ is the $(n-1)$-dimensional Bier ball obtained from $\Delta$ as in \Cref{def: whiskers} (with a slight abuse of notation, we will use the name ``Bier ball'' also when $\Delta$ is the full simplex);
	    \item $R_{\Delta}$ is the $n$-dimensional standard graded Gorenstein $\F$-algebra $\widetilde{\F[\Gamma]} = \F[\Gamma] \ltimes \omega_{\F[\Gamma]}(n-d-1)$, as in \Cref{thm:idealization} (note that, by \Cref{prop: Bier dual}.iv, $\F[\Gamma]$ is level and $a(\F[\Gamma]) = -(n-d) = d-n$);
	    \item given two facets $F_1$ and $F_2$ of $\Delta$, we will denote by $b_{F_1, F_2}$ the binomial $\mathbf{y}^{F_1 \setminus F_2}z_{F_1} - \mathbf{y}^{F_2 \setminus F_1}z_{F_2}$ inside $S[z_F \mid F \in \mathcal{F}(\Delta)]$. Notice that $b_{F_2, F_1} = -b_{F_1, F_2}$, and the case where $F_1$ and $F_2$ coincide yields the zero binomial.
	\end{itemize}
	\end{nota}
	
	\begin{rema} \label{rem:hvector_RDelta}
		The construction of $R_{\Delta}$ shows that if $f=(1,f_0,\dots,f_{d-1})$ is the $f$-vector of a pure simplicial complex, then $(1, f_0+f_{d-1},f_1+f_{d-2},\dots,f_{d-1}+f_0,1)$ is the $h$-vector of a Gorenstein standard graded algebra. In particular, it is an $M$-sequence (see \cite{St_green}). We will see a stronger condition on $\Delta$ for this vector to be the the $h$-vector of a quadratic -- or even Koszul -- Gorenstein algebra.
	\end{rema}

	\begin{prop} \label{prop:RDelta_presentation}
	Let $\Delta$ be a pure simplicial complex of dimension $d-1$. Then \[R_{\Delta} = \frac{S[z_F \mid F \in \mathcal{F}(\Delta)]}{I_{\Gamma} + \mathcal{L} + \left(z_F \mid F \in \mathcal{F}(\Delta)\right)^2},\] where \[\mathcal{L} = ( x_iz_F \mid F \in \mathcal{F}(\Delta),\  i \notin F) + (b_{F_1, F_2} \mid F_1, F_2 \in \mathcal{F}(\Delta), F_1 \neq F_2).\]
	In particular, the defining ideal of $R_{\Delta}$ is generated by
	\begin{itemize}
	    \item $\mathbf{x}^{N_i}$, where $N_i$ is a minimal nonface of $\Delta$,
	    \item $x_iy_i$, where $i$ ranges between $1$ and $n$,
	    \item $z_{F_1}z_{F_2}$, where $F_1$ and $F_2$ are (possibly coincident) facets of $\Delta$,
	    \item $x_iz_F$, where $F$ is a facet of $\Delta$ and $i \notin F$,
	    \item $b_{F_1, F_2}$, where $F_1$ and $F_2$ are distinct facets of $\Delta$ and $b_{F_1, F_2}$ is the binomial defined in \Cref{notat:main},
	\end{itemize}
	and such a presentation does not depend on the characteristic of the field $\F$.
	\end{prop}
	\begin{proof}
	    Since the defining ideal of the claimed presentation of $R_{\Delta}$ involves only monomials and binomials of the form $\mathbf{m} - \mathbf{m'}$, such a presentation must hold in every characteristic, see e.g.~\cite[Lemma A.1.(iii)]{DKoszul}.
	    
	    By Theorem \ref{thm:idealization}.ii, in order to find a presentation for $R_{\Delta}$ it is enough to compute the syzygies of $\omega_{\mathbb{F}[\Gamma]}$ as an $S$-module. (Since we are quotienting by $I_{\Gamma}$ in \eqref{eq: idealization_presentation}, we are actually investigating the $\F[\Gamma]$-syzygies of $\omega_{\F[\Gamma]}$, as $\mathrm{Syz}_1^{\F[\Gamma]}(\omega_{\F[\Gamma]}) \cong \mathrm{Syz}_1^S(\omega_{\F[\Gamma]})/I_{\Gamma}\mathrm{Syz}_1^S(\omega_{\F[\Gamma]})$.) 
	    By Proposition \ref{prop: Bier dual}.iv and by the usual Alexander duality (see e.g.~\cite[Proposition 1.37]{MS}), the canonical module $\omega_{\mathbb{F}[\Gamma]}$ can be identified with the monomial ideal of $\F[\Gamma] = S/I_{\Gamma}$ generated by $\{\overline{\mathbf{y}^{F^c}} : F \text{ facet of }\Delta\}$. Here and until the end of the proof, we will use the shorthand $F^c$ to denote $[n] \setminus F$, the complement of $F$, while the overline denotes the class of $\mathbf{y}^{F^c}$ in $\mathbb{F}[\Gamma]$.
	    Hence, the purity of $\Delta$ implies that all such generators have the same degree $n-d$. Calling $F_1, \ldots, F_r$ the facets of $\Delta$, our task is to compute the kernel of the map of $S$-modules \[\begin{split}S(-n+d)^{|\mathcal{F}(\Delta)|} &\to S/I_{\Gamma}\\ \basis_{F_i} &\mapsto \overline{\mathbf{y}^{F_i^c}}\end{split},\] where each $\basis_{F_i}$ has multidegree $F_i^c$. To do so, we follow \cite[Remark 2.5.6]{Singular}: let $N_1, \ldots, N_s$ be the minimal nonfaces of $\Delta$ and let $A$ be the $1$-by-$(r+n+s)$ matrix
	    \[A = (\mathbf{y}^{F_1^c} \ldots \mathbf{y}^{F_r^c} \mid x_1y_1 \ldots x_ny_n \mid \mathbf{x}^{N_1} \ldots \mathbf{x}^{N_s}).\]
	    Ignoring the degrees, the matrix $A$ represents a map $S^{r+n+s} \to S$. The $S$-module obtained by projecting the syzygy module $\mathrm{Syz}^S(A)$ on its first $r$ coordinates is precisely $\mathrm{Syz}^S_1(\omega_{\F[\Gamma]})$.
	    
	    Since all entries of $A$ are monomials, Schreyer's theorem \cite[Theorem 15.10]{Eis} implies that computing syzygies between such monomials is enough to seize $\mathrm{Syz}^S(A)$ and hence $\mathrm{Syz}^S_1(\omega_{\F[\Gamma]})$. Let $a$ and $b$ be two of the entries in $A$. If $a$ and $b$ both belong to $I_{\Gamma}$, the syzygy between them will live in the last $n+s$ components of $S^{r+n+s}$ and will hence not matter for our final aim. From now on, assume thus that $a = \mathbf{y}^{F_i^c}$ for some $i$. 
	    \begin{itemize}
	        \item If $b \in I_{\Gamma}$ and $\gcd(a, b) = 1$, then the syzygy between $a$ and $b$ will vanish after projecting and taking the quotient modulo $I_{\Gamma}$. In particular, this syzygy will not be involved in the presentation of $R_{\Delta}$.
	        \item If $b \in I_{\Gamma}$ and $\gcd(a, b) > 1$, then it must be that $b = x_jy_j$ and $y_j$ divides $\mathbf{y}^{F_i^c}$, i.e.~$j \notin F_i$. This gives rise to the syzygy $x_j\basis_{F_i}$.
	        \item Finally, assume $b = \mathbf{y}^{F_j^c}$ for some $j \neq i$. Then, noting that $F_j^c \setminus F_i^c = F_i \setminus F_j$, we obtain the syzygy $\mathbf{y}^{F_i \setminus F_j}\basis_{F_i} - \mathbf{y}^{F_j \setminus F_i}\basis_{F_j}$ and this ends the proof.
	    \end{itemize}
	\end{proof}
	
	As an application, we determine when $R_{\Delta}$ is quadratic.

	\begin{prop} \label{prop:superlevel}
	    Let $\Delta$ be a pure flag simplicial complex. Then $R_{\Delta}$ is quadratic if and only if $\Delta$ is $(S_2)$.
	\end{prop}
	\begin{proof}
	    Proceeding as in the proof of \Cref{prop:RDelta_presentation}, one sees that $\omega_{\F[\Gamma]}$ has first linear $\F[\Gamma]$-syzygies if and only if the Alexander dual ideal of $\Delta$ has first linear $\F[\mathbf{y}]$-syzygies. By \Cref{thm:ERTY}, this happens precisely when $\Delta$ is $(S_2)$.
	\end{proof}
	
	\Cref{prop:superlevel} will be vastly generalized in the next section: see \Cref{thm: s_k is k-linear}.
	
	\begin{prop}\label{prop:uni gb}
	The collection $\mathcal{U}$ of monomials and binomials listed in \Cref{prop:RDelta_presentation} is a universal Gr\"obner basis for the ideal defining $R_{\Delta}$ in any characteristic.
    \end{prop}
	\begin{proof}
	    \Cref{prop:RDelta_presentation} says that the defining ideal of $R_{\Delta}$ admits a generating set $\mathcal{G}$ consisting of monomials and binomials of the form $\mathbf{m}-\mathbf{m'}$. Any Gr\"obner basis obtained by applying Buchberger's algorithm to the generating set $\mathcal{U}$ will have the same property, and will hence be characteristic-independent.
	    
	    It is then enough to show that, for any term order $<$, any possible $S$-polynomial reduces to zero with respect to $\mathcal{U}$. The only interesting situation arises when we consider the $S$-polynomial of two distinct binomials $b$ and $b'$, as all the other generators are monomial. If the leading terms of $b$ and $b'$ do not share the same $z$-variable, then $S(b, b')$ has degree two in the $z$-variables and hence reduces to zero with respect to $\mathcal{U}$.
	    
	    We can hence assume that $b = b_{F, G_1}$, $b' = b_{F, G_2}$, $\init_<(b_{F, G_1}) = \mathbf{y}^{F \setminus G_1}z_F$ and $\init_<(b_{F, G_2}) = \mathbf{y}^{F \setminus G_2}z_F$ (where $F$, $G_1$ and $G_2$ are three distinct facets of $\Delta$).
	    
	    We then get that 	    \begin{align*}
	        S(b, b') &= - \mathbf{y}^{(F \setminus G_2) \setminus (F \setminus G_1)}\mathbf{y}^{G_1 \setminus F}z_{G_1} + \mathbf{y}^{(F \setminus G_1) \setminus (F \setminus G_2)}\mathbf{y}^{G_2 \setminus F}z_{G_2}\\
	        &=- \mathbf{y}^{(F \setminus G_2) \cap G_1}\mathbf{y}^{G_1 \setminus F}z_{G_1} + \mathbf{y}^{(F \setminus G_1) \cap G_2}\mathbf{y}^{G_2 \setminus F}z_{G_2}\\
	        &= \mathbf{y}^{(G_1 \cap G_2) \setminus F}b_{G_2, G_1}
	    \end{align*}
	    and, since $b_{G_2, G_1} \in \mathcal{U}$, we are done.
	\end{proof}
	
\begin{exam}\label{ex: section 4}
    Let $\Delta$ be the pure flag simplicial complex with facet list $\{123, 234, 345\}$. The binomials arising from comparing facets of $\Delta$ are
        \[
            \begin{split}
            b_{123, 234} &= y_1z_{123}-y_4z_{234},\\
            b_{234, 345} &= y_2z_{234}-y_5z_{345},\\
            b_{123, 345} &= y_1y_2z_{123}-y_4y_5z_{345}.
            \end{split}
        \]
    However, the binomial $b_{123, 345}$ is redundant, since it can be written as $y_2b_{123, 234}+y_4b_{234, 345}$. 
    Therefore, $R_{\Delta}$ can be presented as the quotient of $\F[x_1,\dots,x_5,y_1,\dots,y_5,z_{123},z_{234},z_{345}]$ by the ideal
    \begin{align*}
        (& x_1x_4,x_1x_5,x_2x_5,x_1y_1,x_2y_2,x_3y_3,x_4y_4,x_5y_5, z_{123}^2,z_{234}^2,z_{345}^2,\\
        &z_{123}z_{234},z_{123}z_{345},z_{234}z_{345},x_4z_{123}, x_5z_{123}, x_1z_{234}, x_5z_{234}, x_1z_{345}, x_2z_{345},\\
        &y_4z_{234}-y_1z_{123}, y_5z_{345}-y_2z_{234}).
    \end{align*}
    Note this agrees with \Cref{prop:superlevel}: being Cohen--Macaulay, $\Delta$ is also $(S_2)$.
\end{exam}

	
\section{\texorpdfstring{When is $R_{\Delta}$ Koszul?}{When is RDelta Koszul?}}
We now wish to investigate the Koszul property of the $\F$-algebra $R_{\Delta}$. The main result, whose proof is postponed to the end of the section, is the following:
\begin{theorem}\label{thm: s_k is k-linear}
		Let $\Delta$ be a flag $(d-1)$-dimensional simplicial complex and let $1 \leq k\leq d$. $\F$ has a resolution as an $R_{\Delta}$-module which is linear for $k$ steps if and only if $\Delta$ satisfies property $(S_{k})$. Moreover, if $\F$ has an $R_{\Delta}$-resolution which is linear for $d$ steps, then it has a linear $R_{\Delta}$-resolution.
\end{theorem}
\begin{rema}
    The homological behaviour of $R_{\Delta}$ is especially remarkable, since for a general standard graded quadratic $\F$-algebra $Q$ the linearity of the $Q$-resolution of $\F$ might fail at an arbitrarily high homological position \cite{Roos}.
\end{rema}
As a corollary of \Cref{thm: s_k is k-linear} we obtain the following:
\begin{coro}\label{cor: koszul_iff_cm}
		Let $\Delta$ be a flag simplicial complex. Then $R_{\Delta}$ is Koszul over $\F$ if and only if $\Delta$ is Cohen--Macaulay over $\F$.
	\end{coro}

\subsection{The generalized Koszul complex for a quadratic Stanley--Reisner ring}

	Let $\Sigma$ be a flag simplicial complex and consider its Stanley--Reisner ring $\F[\Sigma] = \F[z_i \mid \{i\} \in \Sigma]/I_{\Sigma}$, which is a Koszul algebra. It was already known to Fr\"oberg \cite{Fr} how to explicitly describe a minimal $\F[\Sigma]$-resolution of the residue field $\F$: we will briefly describe such a \emph{generalized Koszul complex} here, following the treatment in \cite[Section 8]{MP}.
	
	The Koszul algebra $\F[\Sigma]$ admits a \emph{Koszul dual algebra} $\F[\Sigma]^{!}$, which is obtained by quotienting the noncommutative polynomial ring $\F\langle Z_i \mid \{i\} \in \Sigma\rangle$ by the relations 
	
	\begin{equation} \label{eq:koszul_dual}
	\begin{split}
	Z_i^2 &\text{ for every vertex }i \text{ of } \Sigma\\ 
	Z_iZ_j + Z_jZ_i &\text{ for every edge }\{i,j\} \text{ of }\Sigma.
	\end{split}
	\end{equation}
	
	In particular, given a noncommutative monomial in the $Z$-variables, two consecutive distinct variables $v$ and $v'$ can anticommute unless $vv' \in I_{\Sigma}$.
	
	\begin{nota}
	In what follows we will think of noncommutative monomials $\mathbf{w}$ in the $Z$-variables as \emph{words}, and will denote by $[\mathbf{w}]$ the equivalence class of $\mathbf{w}$ with respect to the relations \eqref{eq:koszul_dual} (sometimes writing $[-\mathbf{w}]$ to denote $-[\mathbf{w}]$). We will write $\mathbf{w} \| v$ for the word obtained by adding the letter $v$ at the end of the word $\mathbf{w}$. Finally, for any nonzero $[\mathbf{w}]$, the support $\mathrm{supp}([\mathbf{w}])$ will be the signless commutative monomial obtained by multiplying together the $z$-variables corresponding to the letters of any representative of $[\mathbf{w}]$.
	\end{nota}
	
	\begin{defi}
    The generalized Koszul complex $(\mathbb{GK}_{\bullet}(\F[\Sigma]), \partial)$ is the chain complex of free $\F[\Sigma]$-modules given by the following data:
	\begin{itemize}
	    \item $\mathbb{GK}_j(\F[\Sigma]) = \F[\Sigma] \otimes_{\F} \F[\Sigma]^!_j$;
	    \item if $j > 0$ and $\mathbf{w} = Z_{i_1}Z_{i_2}\ldots Z_{i_j}$ is a $j$-letter word, the differential $\partial$ is given by 
	    \begin{equation} \label{eq:GK_diff}
	    \partial(1 \otimes [\mathbf{w}]) = \sum_{k \in \mathrm{head}({\mathbf{w}})}(-1)^{k-1}z_{i_k} \otimes [\mathbf{w} \setminus \{Z_{i_k}\}],
	    \end{equation}
	    where $\mathbf{w} \setminus \{Z_{i_k}\}$ is the $(j-1)$-letter word obtained by erasing the letter $Z_{i_k}$ from $\mathbf{w}$ and $\mathrm{head}(\mathbf{w})$ is the set of those indices $k$ for which there exists a representative of $[\mathbf{w}]$ with $Z_{i_k}$ as its first letter. One can check that such a differential is indeed well-defined.
	\end{itemize}
	\end{defi}
	
	\begin{exam}
	    Let $\Sigma = \{12, 23, 1, 2, 3, \varnothing\}$.
	    
	    Then an $\F$-basis for $\F[\Sigma]^{!}_2$ is $\{[Z_1Z_2], [Z_1Z_3], [Z_2Z_3], [Z_3Z_1]\}$, whereas an $\F$-basis for $\F[\Sigma]^{!}_3$ is $\{[Z_1Z_2Z_3], [Z_1Z_3Z_1], [Z_2Z_3Z_1], [Z_3Z_1Z_3]\}$.
	    
	    Note that, for instance, $[Z_1^2] = 0$ and $[Z_2Z_3Z_1] = -[Z_3Z_2Z_1] = [Z_3Z_1Z_2]$, but $[Z_2Z_3Z_1] \neq [Z_2Z_1Z_3]$. If $\mathbf{w}=Z_2Z_3Z_1$, we have that $\mathrm{head}({\mathbf{w}})=\{1,2\}$ and \[\partial(1\otimes[Z_2Z_3Z_1])= z_2 \otimes [Z_3Z_1] - z_3 \otimes [Z_2Z_1] = z_2 \otimes [Z_3Z_1] + z_3 \otimes [Z_1Z_2].\]
	    
	    The matrix describing the map $\mathbb{GK}_3(\F[\Sigma]) \xrightarrow{\partial} \mathbb{GK}_2(\F[\Sigma])$ in the proposed basis is
	    \[
	    \bordermatrix
	    {   ~ & 1 \otimes [Z_1Z_2Z_3] & 1 \otimes [Z_1Z_3Z_1] & 1 \otimes [Z_2Z_3Z_1] & 1 \otimes [Z_3Z_1Z_3]\cr
	        1 \otimes [Z_1Z_2] & 0 & 0 & z_3 & 0\cr
	        1 \otimes [Z_1Z_3] & -z_2 & 0 & 0 & z_3\cr
	        1 \otimes [Z_2Z_3] & z_1 & 0 & 0 & 0\cr
	        1 \otimes [Z_3Z_1] & 0 & z_1 & z_2 & 0\cr}.
	    \]
	\end{exam}

\subsection{Homological properties of Stanley--Reisner rings of flag Bier balls}
    
    For the rest of this section we fix a pure flag simplicial complex $\Delta$ and the associated Bier ball $\Gamma$.
    
	In the following we will be interested in $\mathbb{N}^{2n}$-graded objects; we will reserve the word ``multidegree'' to denote either a vector in $\mathbb{N}^{2n}$ or the associated monomial in $x_1, \ldots, x_n, y_1, \ldots, y_n$.
	
	When working with multidegrees, we shall distinguish between two types of variables.
	
	\begin{defi}
		Let $\mathbf{m}$ be a multidegree. A variable $v|\mathbf{m}$ is \emph{red} with respect to $\mathbf{m}$ (and $\Gamma$) if there exists $v'|\mathbf{m}$ such that $vv'\in I_{\Gamma}$. Observe that necessarily $v\neq v'$. A variable $v|\mathbf{m}$ which is not red is called \emph{blue}. We say that a multidegree is blue if all the variables in its support are blue, and red otherwise.
	\end{defi}
	
	\begin{rema}
	By definition, blue multidegrees correspond bijectively to monomials of $\F[\Gamma]$. In particular, any $\mathbb{Z}^{2n}$-graded ideal of $\F[\Gamma]$ is generated by a collection of blue monomials. \end{rema}
	
	We now establish two technical lemmas about the homological behaviour of $\mathbb{Z}^{2n}$-graded ideals in $\F[\mathbf{x}, \mathbf{y}]$ and $\F[\Gamma]$: these results will be crucial for the proof of \Cref{thm: s_k is k-linear}.
	
	\begin{lem}[Blue Lemma] \label{lemma:blue}
Let $\mathbf{m}$ be a blue multidegree, let $\mathcal{M}$ be a collection of blue monomials, and denote by $I^{\mathcal{M}}$ (respectively, $J^{\mathcal{M}}$) the $\mathbb{Z}^{2n}$-graded ideal of $\F[\mathbf{x}, \mathbf{y}]$ (respectively, $\mathbb{F}[\Gamma]$) generated by $\mathcal{M}$. Then:
    \begin{enumerate}
        \item if $\mathbf{n}$ divides $\mathbf{m}$, then $\mathbf{n}$ is also a blue multidegree;
        \item the $\F$-vector spaces $I^{\mathcal{M}}_{\mathbf{m}}$ and $J^{\mathcal{M}}_{\mathbf{m}}$ are either both one-dimensional or both $\{0\}$;
        \item $\mathbb{GK}_{\bullet}(\F[\Gamma])_{\mathbf{m}}= \mathbb{K}_{\bullet}(\mathbf{x}, \mathbf{y})_{\mathbf{m}}$, where $\mathbb{K}_{\bullet}(\mathbf{x}, \mathbf{y})$ is the usual Koszul complex on the variables $\mathbf{x}$ and $\mathbf{y}$;
        \item $\beta^{\mathbb{F}[\Gamma]}_{i,\mathbf{m}}(J^{\mathcal{M}}) = \beta^{\F[\mathbf{x}, \mathbf{y}]}_{i, \mathbf{m}}(I^{\mathcal{M}})$ for every $i \in \mathbb{N}$.
    \end{enumerate}
\end{lem}
    \begin{proof}
    \begin{enumerate}
        \item This is a direct consequence of the definition of blue multidegree.
        
        \item Both $J^{\mathcal{M}}_{\mathbf{m}}$ and $I^{\mathcal{M}}_{\mathbf{m}}$ can be at most one-dimensional; if they differ, then it must be that $\mathbf{m} \in I_{\Gamma}$, but this cannot happen precisely because $\mathbf{m}$ is blue.
        
        \item Fix $j \in \mathbb{N}$ and consider \[\mathbb{GK}_{j}(\F[\Gamma])_{\mathbf{m}} = (\F[\Gamma] \otimes \F[\Gamma]^{!}_{j})_{\mathbf{m}} = \bigoplus_{\substack{\mathbf{n} \mid \mathbf{m}\\ |\mathbf{n}|=j}}{\F[\Gamma]}_{\frac{\mathbf{m}}{\mathbf{n}}} \otimes \F[\Gamma]^{!}_{\mathbf{n}}.\]
        
Pick a multidegree $\mathbf{n}$ dividing $\mathbf{m}$ and such that $|\mathbf{n}|=j$; by part i, $\mathbf{n}$ is blue. For any word $\mathbf{w}$ with support $\mathbf{n}$, one has that $[\mathbf{w}] \neq 0$ if and only if $\mathbf{n}$ is squarefree: otherwise, since all letters can anticommute, we would be able to put next to each other two occurrences of the same letter, causing the whole class to vanish. We can hence assume that $\mathbf{n}$ is squarefree; in particular, since $x_hy_h \in I_{\Gamma}$ for every $1 \leq h \leq n$, we can write $\mathbf{n}$ as $z_{i_1}z_{i_2}\ldots z_{i_j}$, where $i_1 < i_2 < \ldots < i_j$ and each $z_{i_k}$ is either $x_{i_k}$ or $y_{i_k}$. All words $\mathbf{w}$ with support $\mathbf{n}$ belong to the same class (up to a global sign), since the order of the letters does not really matter in this case: we will denote such class by $z_{i_1}z_{i_2}\ldots z_{i_j}$. Such a characterization proves that $\mathbb{GK}_{j}(\F[\Gamma])_{\mathbf{m}}$ coincides with $\mathbb{K}_{j}(\mathbf{x}, \mathbf{y})_{\mathbf{m}}$. As for the differential, since $\mathrm{head}(z_{i_1}z_{i_2}\ldots z_{i_j}) = \{1, 2, \ldots, j\}$, we get \[\partial\left(\frac{\mathbf{m}}{\mathbf{n}} \otimes \mathbf{1}_{z_{i_1}z_{i_2}\ldots z_{i_j}}\right) = \sum_{k=1}^{j}(-1)^{k-1}\frac{z_{i_k}\mathbf{m}}{\mathbf{n}} \otimes \mathbf{1}_{z_{i_1}z_{i_2}\ldots \hat{z}_{i_k}\ldots z_{i_j}},\]
        i.e.~the differential of the (degree $\mathbf{m}$ part of the) usual Koszul complex on the variables $\mathbf{x}$ and $\mathbf{y}$.
        
        \item It is enough to show that the chain complexes of $\F$-vector spaces $(J^{\mathcal{M}} \otimes \mathbb{GK}_{\bullet}(\F[\Gamma]))_{\mathbf{m}}$ and $(I^{\mathcal{M}} \otimes \mathbb{K}_{\bullet}(\mathbf{x}, \mathbf{y}))_{\mathbf{m}}$ coincide.
        
        For every $j \in \mathbb{N}$ one has that \[(J^{\mathcal{M}} \otimes \mathbb{GK}_{j}(\F[\Gamma]))_{\mathbf{m}} = \bigoplus_{\mathbf{n} \mid \mathbf{m}}J^{\mathcal{M}}_{\frac{\mathbf{m}}{\mathbf{n}}} \otimes (\mathbb{GK}_{j}(\F[\Gamma]))_{\mathbf{n}};\] by part i, both $\mathbf{n}$ and $\frac{\mathbf{m}}{\mathbf{n}}$ are blue multidegrees. By parts ii and iii, one then has that $J^{\mathcal{M}}_{\frac{\mathbf{m}}{\mathbf{n}}} = I^{\mathcal{M}}_{\frac{\mathbf{m}}{\mathbf{n}}}$ and $\mathbb{GK}_{j}(\F[\Gamma])_{\mathbf{n}}= \mathbb{K}_{j}(\mathbf{x}, \mathbf{y})_{\mathbf{n}}$. Analyzing the differential maps as in the proof of part iii yields the claim.
    \end{enumerate}
    \end{proof}
	
	\begin{lem}[Red Lemma] \label{lemma:red}
	Let $J$ be a $\mathbb{Z}^{2n}$-graded ideal of $\mathbb{F}[\Gamma]$, let $\mathbf{m}$ be a red multidegree and assume that $\beta_{i,\mathbf{m}}(J) \neq 0$ for some $i > 0$. Then there exists a red variable $v$ such that $\beta_{i-1, \frac{\mathbf{m}}{v}}(J) \neq 0$.
	\end{lem}
	\begin{proof}
	Let us start by considering a cycle $z$ which has homological degree $i$, internal multidegree $\mathbf{m}$ and is not a boundary. Since $\Tor_i(J, \F)_{\mathbf{m}} = H_i(J \otimes \mathbb{GK}_{\bullet}(\F[\Gamma]))_{\mathbf{m}},$ such a cycle can be written as 
	\begin{equation} \label{eq: red_lemma_cycle}
	z = \sum_{j \in \mathcal{C}} \lambda_jn_j \otimes [\mathbf{w}_j],
	\end{equation}
	where for each $j$ in the finite set $\mathcal{C}$ one has that $\lambda_j \in \F \setminus \{0\}$, $n_j$ is a monomial in $J$, $[\mathbf{w}_j]$ is the nonzero class of a word $\mathbf{w}_j$ with $i$ letters, and $n_j \cdot \mathrm{supp}([\mathbf{w}_j]) = \mathbf{m}$. Note that each word $\mathbf{w}_j$ in \eqref{eq: red_lemma_cycle} must contain at least one red letter $v$. If this were not the case for some $\mathbf{w}_j$, then any two red variables $v$ and $\bar{v}$ such that $v\bar{v} \in I_{\Gamma}$ would both divide $n_j$, causing the whole term to vanish.
	
	Since blue letters anticommute with every other letter, there must be at least a red letter $v_1$ that appears as the last letter of a representative of some $[\mathbf{w}_j]$, up to sign. Let us fix such a $v_1$. By possibly tweaking the sign of some $\lambda_j$, we can assume without loss of generality that \begin{equation}z = \sum_{j \in \mathcal{A}}\lambda_j n_j \otimes [\tilde{\mathbf{w}}_j \| v_1] + \sum_{j \in \mathcal{B}} \lambda_jn_j \otimes [\mathbf{w}_j] =: z_{v_1} + z',\end{equation}
	where each $\tilde{\mathbf{w}}_j$ has $i-1$ letters, $\mathcal{A} \sqcup \mathcal{B} = \mathcal{C}$, and the nonempty set $\mathcal{A}$ contains all the indices for which $v_1$ can travel to the end of the associated word.
	
	\textbf{Claim:} $z_{v_1}$ is a cycle (and hence so is $z'$).
	
	To see this, note that all the basis elements appearing in $\partial (1 \otimes [\tilde{\mathbf{w}}_j \| v_1])$ must end in $v_1$, except possibly for those words $\tilde{\mathbf{w}}_j \| v_1$ where $v_1$ is free to navigate to the front. However, this last instance can happen only if the support of $\tilde{\mathbf{w}}_j$ contains no variable $\bar{v}_1$ with $v_1\bar{v}_1\in I_{\Gamma}$; then, since $\bar{v}_1$ divides $\mathbf{m}$, it must be that $\bar{v}_1$ divides $n_j$. But then, when $v_1$ exits the word, it multiplies $\bar{v}_1$ and vanishes. Hence, $\partial z_{v_1}$ can be written as a combination of words ending in $v_1$. On the other hand, $\partial{z'}$ can contain no such words. But then, since $\partial z = 0$, it must be that $\partial z_{v_1} = \partial z' = 0$. Note that both of these new cycles have again homological degree $i$ and internal multidegree $\mathbf{m}$. If $z_{v_1}$ is not a boundary, we keep it and we stop. If instead $z_{v_1}$ is a boundary, then $z'$ cannot be so (otherwise the original $z$ would also be a boundary) and we can iterate the same procedure as before, finding a new red variable $v_2 \neq v_1$ and writing $z'=z'_{v_2} + z''.$ This process can be repeated only finitely many times and will yield the desired non-boundary cycle.
	
	From now on we can hence assume without loss of generality that \begin{equation}z = \sum_{j \in \mathcal{A}}\lambda_j n_j \otimes [\tilde{\mathbf{w}}_j \| v].\end{equation}
	
	Now let \[\hat{z} := \sum_{j \in \mathcal{A}}\lambda_j n_j \otimes [\tilde{\mathbf{w}}_j].\]
	Reasoning as in the proof of the claim above, it follows that $\hat{z}$ is a cycle; moreover, if there existed $\hat{y}$ such that $\partial \hat{y} = \hat{z}$, then $z$ would also be a boundary (appending $v$ at the end of each word in $\hat{y}$). By definition, the cycle $\hat{z}$ has homological degree $i-1$ and internal multidegree $\frac{\mathbf{m}}{v}$, and this concludes the proof.
	\end{proof}

    \begin{coro} \label{cor:same_reg}
    Let $\mathcal{M}$ be a collection of blue monomials, and denote by $I^{\mathcal{M}}$ (respectively, $J^{\mathcal{M}}$) the $\mathbb{Z}^{2n}$-graded ideal of $\F[\mathbf{x}, \mathbf{y}]$ (respectively, $\mathbb{F}[\Gamma]$) generated by $\mathcal{M}$. Let $k$ and $h$ be two nonnegative integers. Then:
    \begin{enumerate}
        \item for every $i$, $j$ for which $\beta_{i,j}^{\F[\Gamma]}(J^{\mathcal{M}}) \neq 0$ there exists $0 \leq h \leq i$ for which $\beta_{i-h,j-h}^{\F[\mathbf{x}, \mathbf{y}]}(I^{\mathcal{M}}) \neq 0$; 
        
        \item if $\beta^{\F[\mathbf{x}, \mathbf{y}]}_{i,j}(I^{\mathcal{M}}) = 0$ for every $0 \leq i \leq k$ and $j > i+h$, then $\beta^{\F[\Gamma]}_{i,j}(J^{\mathcal{M}}) = 0$ for every $0 \leq i \leq k$ and $j > i+h$; in particular, if the $\F[\mathbf{x}, \mathbf{y}]$-resolution of $I^{\mathcal{M}}$ is linear for $k$ steps, so is the $\F[\Gamma]$-resolution of $J^{\mathcal{M}}$; 
        
        \item $\reg_{\mathbb{F}[\Gamma]}(J^{\mathcal{M}}) \leq \reg_{\F[\mathbf{x}, \mathbf{y}]}(I^{\mathcal{M}}).$
        \end{enumerate}
        
        If moreover all the monomials in $\mathcal{M}$ contain only $y$-variables, then:
        \begin{enumerate}
        \setcounter{enumi}{3}
        \item if $\beta^{\F[\mathbf{x}, \mathbf{y}]}_{i,j}(I^{\mathcal{M}}) \neq 0$, then $\beta_{i,j}^{\F[\Gamma]}(J^{\mathcal{M}}) \neq 0$;
        
        \item $\beta^{\F[\mathbf{x}, \mathbf{y}]}_{i,j}(I^{\mathcal{M}}) = 0$ for every $0 \leq i \leq k$ and $j > i+h$ if and only if $\beta^{\F[\Gamma]}_{i,j}(J^{\mathcal{M}}) = 0$ for every $0 \leq i \leq k$ and $j > i+h$;
        
        \item $\reg_{\mathbb{F}[\Gamma]}(J^{\mathcal{M}}) = \reg_{\F[\mathbf{x}, \mathbf{y}]}(I^{\mathcal{M}}).$ In particular, $J^{\mathcal{M}}$ has a linear $\F[\Gamma]$-resolution if and only if $I^{\mathcal{M}}$ has a linear $\F[\mathbf{x}, \mathbf{y}]$-resolution.
    \end{enumerate}
    \end{coro}
    \begin{proof}
    \begin{itemize}
        \item[i.] Pick $i$ and $j$ such that $\beta^{\F[\Gamma]}_{i, j}(J^{\mathcal{M}}) \neq 0$, and choose a multidegree $\mathbf{m}$ such that $|\mathbf{m}| = j$ and $\beta^{\F[\Gamma]}_{i, \mathbf{m}}(J^{\mathcal{M}}) \neq 0$. If $\mathbf{m}$ is blue, \Cref{lemma:blue}.iv yields the claim with $h=0$. If instead $\mathbf{m}$ is red, we know by \Cref{lemma:red} that we can decrease the homological index from $i$ to $i-1$ while staying in the $(j-i)$-th linear strand of the $\F[\Gamma]$-resolution of $J^{\mathcal{M}}$. This can be repeated until we hit a blue multidegree, which happens at the latest when we reach the zeroth homological degree, since all generators of $J^{\mathcal{M}}$ are blue. Then applying \Cref{lemma:blue}.iv finishes the proof.
        
        \item[ii.--iii.] These statements follow from part i.
        
        \item[iv.] Pick $i$ and $j$ such that $\beta^{\F[\mathbf{x}, \mathbf{y}]}_{i, j}(I^{\mathcal{M}}) \neq 0$, and choose a multidegree $\mathbf{m}$ such that $|\mathbf{m}| = j$ and $\beta^{\F[\mathbf{x}, \mathbf{y}]}_{i, \mathbf{m}}(I^{\mathcal{M}}) \neq 0$. Since no $x$-variables are involved in any minimal generator of $I^{\mathcal{M}}$, one has that $\beta^{\F[\mathbf{x}, \mathbf{y}]}_{i, \mathbf{m}}(I^{\mathcal{M}})$ can be nonzero only if $\mathbf{m}$ is a monomial in the $y$-variables; however, any such $\mathbf{m}$ must be blue, as by construction the only variable $v$ with $vy_i\in I_{\Gamma}$ is $v=x_i$. Applying Lemma \ref{lemma:blue}.iv then yields that $\beta^{\F[\Gamma]}_{i, \mathbf{m}}(J^{\mathcal{M}}) = \beta^{\F[\mathbf{x}, \mathbf{y}]}_{i, \mathbf{m}}(I^{\mathcal{M}})$, and the claim follows.
        
        \item[v.--vi.] These statements follow from combining parts i and iv.
    \end{itemize}    
        
    \end{proof}
    
    \begin{rema} \label{rem:echo}
    Under the hypotheses of \Cref{cor:same_reg}.iv--vi, the Betti table of $J^{\mathcal{M}}$ is a ``horizontal echo'' of the (finite) Betti table of $I^{\mathcal{M}}$.
    
    As an example, take $\F[\Gamma] = \F[\mathbf{x}, \mathbf{y}]/(x_1x_2, x_1x_3, x_1y_1, x_2y_2, x_3y_3)$ and $\mathcal{M} = \{y_1y_2, y_2^2y_3^2, y_3^4\}$. Using Macaulay2 \cite{M2}, we find out that the Betti tables of $I^{\mathcal{M}}$ and $J^{\mathcal{M}}$ look as follows:
    
    \begin{table}[h!]
    \centering
       \begin{minipage}{0.25\textwidth}
    $\begin{array}{r | c c c}
       &0&1&2\\
       \hline\text{2}&\mathbf{1}&\text{.}&\text{.}\\\text{3}&\text
       {.}&\text{.}&\text{.}\\
       \text{4}&\mathbf{2}&\mathbf{1}&\text{.}\\\text{5}&\text{.}&\mathbf{2}&\mathbf{1}\\
       \end{array}$
       \end{minipage}
       \begin{minipage}{0.43\textwidth}
       $\begin{array}{r|c c c c c c c c}
       &0&1&2&3&4&5&6&7\\
       \hline\text{2}&\mathbf{1}&2&5&13&34&89&233&610\\
       \text{3}&\text{.}&\text{.}&\text{.}&\text{.}&\text{.}&\text{.}&\text{.}&\text{.}\\
       \text{4}&\mathbf{2}&\mathbf{4}&10&26&68&178&466&1220\\
       \text{5}&\text{.}&\mathbf{2}&\mathbf{6}&16&42&110&288&754\\
       \end{array}$
     \end{minipage}
     
     \vspace{10pt}
     
     \caption{The Betti table of $I^{\mathcal{M}}$ (left) and the beginning of the Betti table of $J^{\mathcal{M}}$ (right) from \Cref{rem:echo}.}
     \end{table}
    \end{rema}
    
    \begin{rema}
        Note that the statements of \Cref{cor:same_reg}.iv--vi can indeed fail when the elements of $\mathcal{M}$ are not in the $y$-variables only.
        
        For instance, take $\F[\Gamma] = \F[\mathbf{x}, \mathbf{y}]/(x_1x_3, x_1y_1, x_2y_2, x_3y_3, x_4y_4)$ and $\mathcal{M} = \{x_1x_2, x_3x_4\}$. Since $I^{\mathcal{M}} = (x_1x_2, x_3x_4)$ is a complete intersection in $\F[\mathbf{x}, \mathbf{y}]$, one has that $\reg_{\F[\mathbf{x}, \mathbf{y}]}(I^{\mathcal{M}}) = 3$. However, since $J^{\mathcal{M}}$ has linear quotients with respect to the Koszul filtration consisting of all subsets of variables of $\F[\Gamma]$, then $J^{\mathcal{M}}$ has a $2$-linear $\F[\Gamma]$-resolution \cite[Lemma 17]{CDR}; hence, $2 = \reg_{\mathbb{F}[\Gamma]}(J^{\mathcal{M}}) < \reg_{\F[\mathbf{x}, \mathbf{y}]}(I^{\mathcal{M}}) = 3$.
    \end{rema}
    
    We are finally ready to prove \Cref{thm: s_k is k-linear}.\\
	\begin{proof}(of \Cref{thm: s_k is k-linear})
	    Let $\mathcal{M} := \{\mathbf{y}^{[n] \setminus F} \mid F \text{ facet of }\Delta\}$, and let $I^{\mathcal{M}} \subseteq \F[\mathbf{x}, \mathbf{y}]$, $J^{\mathcal{M}} \subseteq \F[\Gamma]$ be as in \Cref{cor:same_reg}. Note that $I^{\mathcal{M}}$ is the extension of the Alexander dual ideal of $\Delta$ to a polynomial ring containing also $x$-variables. By \Cref{thm:ERTY}, the simplicial complex $\Delta$ is $(S_{k})$ if and only if the Alexander dual ideal of $\Delta$ has an $\F[\mathbf{y}]$-resolution which is linear for $k-1$ steps. 
	    This is equivalent to stating that $I^{\mathcal{M}}$ has an $\F[\mathbf{x}, \mathbf{y}]$-resolution which is linear for $k-1$ steps. By \Cref{cor:same_reg}.v, this is in turn equivalent to stating that $J^{\mathcal{M}}$ has an $\F[\Gamma]$-resolution which is linear for $k-1$ steps. By \Cref{prop: Bier dual}.iv, the ideal $J^{\mathcal{M}}$ coincides with the canonical module $\omega_{\F[\Gamma]}$. Finally, applying \Cref{lem: from omega to F} with $A = \F[\Gamma]$, $M = \omega_{\F[\Gamma]}(-a(\F[\Gamma])-1)$ concludes the proof.
	\end{proof}

	\begin{rema} \label{rem: codim_and_reg}
		In \cite{MSSI} the authors consider the problem of characterizing the pairs of positive integers $(c,r)$ for which there exists a non-Koszul quadratic Gorenstein algebra $R$ with codimension $c$ and regularity $r$. The results of \cite{MSSI}, \cite{MSSII} and \cite{McSe} leave only two cases open, namely $(c,r)=(6,3)$ and $(c,r)=(7,3)$. One might wonder if the results in this section provide examples of non-Koszul quadratic Gorenstein algebras with such invariants. However, this is not the case: if $R_{\Delta}$ has regularity 3, then $\Delta$ must be 1-dimensional, and in this case $\Delta$ satisfies $(S_2)$ if and only if it is Cohen--Macaulay (over any field).
	\end{rema}	

	\subsection{An application: quadratic Gorenstein algebras which are not Koszul in prescribed characteristics} \label{subsec: Koszulness and characteristic}
	We conclude this section with an application of \Cref{cor: koszul_iff_cm}. Given a finite list of prime numbers $P=\{p_1,\dots,p_m\}$, our result can be used to construct quadratic Gorenstein $\mathbb{F}$-algebras which are Koszul if and only if $\text{char}(\mathbb{F})\notin P$. Moreover, the presentations of these algebras will be characteristic-free, see \Cref{prop:RDelta_presentation}.
	
	To do so, it suffices to exhibit a flag simplicial complex which is Cohen--Macaulay over $\mathbb{F}$ if and only if the characteristic of $\mathbb{F}$ is not in $P$. Note that such a complex will automatically be $(S_2)$ (which corresponds to $R_{\Delta}$ having a quadratic presentation, see \Cref{prop:superlevel}), since the $(S_2)$ property is characteristic-free and the complex is Cohen--Macaulay in some characteristic.
	
	We will focus on flag triangulations of $3$-dimensional \emph{lens spaces}. A $3$-dimensional lens space is an orientable $3$-manifold obtained as a quotient of the $3$-sphere by certain rotations (see \cite[Example 2.43]{Hat}). Such spaces are parametrized by two coprime integers $p,q\geq 1$, and $L(p_1,q_1)\cong L(p_2,q_2)$ if and only if $p_1=p_2$ and $q_1=\pm q_1^{\pm 1} \mod{p}$. The reduced integral homology groups of $L(p,q)$ are as follows: $\tilde{H}_0(L(p,q))= \tilde{H}_2(L(p,q))=0$, $\tilde{H}_3(L(p,q))=\mathbb{Z}$ and $\tilde{H}_1(L(p,q))=\mathbb{Z}_p$. As lens spaces are homology $3$-manifolds, the link of any nonempty face in any triangulation $\Delta$ of $L(p,q)$ is a homology sphere. Hence, by Reisner's criterion together with the universal coefficient theorem \cite[Section 3A]{Hat}, $\Delta$ is Cohen--Macaulay over a field $\mathbb{F}$ if and only if $\tilde{H}_1(\Delta;\mathbb{F})=\tilde{H}_1(\Delta;\mathbb{Z})\otimes \mathbb{F}=\mathbb{F}/p\mathbb{F}=0$. In particular, every triangulation of $L(p,q)$ is Cohen--Macaulay over $\mathbb{F}$ if and only if $\text{char}(\mathbb{F})\neq p$. Therefore if $\Delta$ is a flag triangulation of $L(p,q)$, then $R_{\Delta}$ is a quadratic Gorenstein $\mathbb{F}$-algebra which is Koszul if and only if $\text{char}(\mathbb{F})\neq p$. We will generalize this fact by considering connected sums of lens spaces. 
	\begin{prop}\label{prop: lens spaces}
	    Let $P = \{p_1,\dots,p_m\}$ be a set of prime numbers. Let $\Delta$ be any flag triangulation of the connected sum $L(p_1,q_1)\# \cdots \# L(p_m,q_m)$, for some positive integers $q_1,\dots,q_m$. Then $R_{\Delta}$ is a quadratic Gorenstein $\mathbb{F}$-algebra that is Koszul if and only if $\text{char}(\mathbb{F})\notin P$.
	\end{prop}
	\begin{proof}
	    As $\Delta$ is a triangulated $3$-manifold, one has that $\tilde{H}_i(\lk_{\Delta}(F);\mathbb{Z})=0$ for every $\emptyset\neq F\in \Delta$ and for every $i<\dim(F)$. Moreover, we can compute the homology of a connected sum of manifolds by a standard application of the Mayer--Vietoris sequence. This yields that $\tilde{H}_2(\Delta;\mathbb{Z})=0$ and $\tilde{H}_1(\Delta;\mathbb{Z})=\bigoplus_{p\in P} \mathbb{Z}/p\mathbb{Z}$. Thus, we have that $\Delta$ is Cohen--Macaulay over a field $\mathbb{F}$ if and only if $\tilde{H}_1(\Delta;\mathbb{Z})\otimes\mathbb{F}=\bigoplus_{p\in P} \mathbb{F}/p\mathbb{F}=0$, which in turn happens if and only if $\text{char}(\mathbb{F})\notin P$. We conclude by using \Cref{cor: koszul_iff_cm}.
	\end{proof}
	Requiring that $\Delta$ is flag is not restrictive, as any triangulable space has a flag triangulation given, for example, by its barycentric subdivision \cite[III.4]{St_green}. However, these simplicial complexes typically have a lot of vertices, and it is a challenging problem to find vertex-minimal flag triangulations of a given space.
	
	Finally, we observe that \Cref{prop: lens spaces} can be applied verbatim to flag triangulations of higher-dimensional lens spaces, yielding quadratic Gorenstein algebras with the same behaviour as in \Cref{prop: lens spaces} but with an $h$-polynomial of higher degree.
	
	\begin{exam}\label{ex: RP2} Let $\Delta$ be a flag triangulation of $\mathbb{RP}^2$. In \cite{BOWWZZ} the authors construct two such non-isomorphic triangulations with the same $f$-vector $f(\Delta)=(1,11,30,20)$. We let $\Delta$ be the simplicial complex whose list of facets is reported in the left column of \cite[Table 1]{BOWWZZ}. The algebra $R_{\Delta}$ is then presented as $T/J$,  with 
	\begin{align*}
	T = \mathbb{F}[& x_1,\dots,x_{9},x_a,x_{b},y_1,\dots,y_9,y_a,y_{b},z_{145},z_{126},z_{156},z_{237},z_{347},z_{267},\\
	& z_{148}, z_{478},z_{129},
	z_{189},z_{23a}, z_{34a},z_{45a},z_{29a},z_{56b},z_{67b},z_{78b},z_{89b},z_{5ab},z_{9ab}]
	\end{align*}
	and
	    \begingroup
	    \allowdisplaybreaks
	   \begin{align*}
	        J = (& x_{\bsp}z_{145}, x_{\asp}z_{145}, x_{9}z_{145}, x_{8}z_{145}, x_{7}z_{145}, x_{6}z_{145}, x_{3}z_{145}, x_{2}z_{145}, x_{\bsp}z_{126}, x_{\asp}z_{126}, x_{9}z_{126}, x_{8}z_{126},\\ & x_{7}z_{126}, x_{5}z_{126}, x_{4}z_{126}, x_{3}z_{126}, x_{\bsp}z_{156}, x_{\asp}z_{156}, x_{9}z_{156}, x_{8}z_{156}, x_{7}z_{156}, x_{4}z_{156}, x_{3}z_{156}, x_{2}z_{156},\\ & x_{\bsp}z_{237}, x_{\asp}z_{237}, x_{9}z_{237}, x_{8}z_{237}, x_{6}z_{237}, x_{5}z_{237}, x_{4}z_{237}, x_{1}z_{237}, x_{\bsp}z_{347}, x_{\asp}z_{347}, x_{9}z_{347}, x_{8}z_{347},\\ & x_{6}z_{347}, x_{5}z_{347}, x_{2}z_{347}, x_{1}z_{347}, x_{\bsp}z_{267}, x_{\asp}z_{267}, x_{9}z_{267}, x_{8}z_{267}, x_{5}z_{267}, x_{4}z_{267}, x_{3}z_{267}, x_{1}z_{267},\\
	        & x_{\bsp}z_{148}, x_{\asp}z_{148}, x_{9}z_{148}, x_{7}z_{148}, x_{6}z_{148}, x_{5}z_{148}, x_{3}z_{148}, x_{2}z_{148}, x_{\bsp}z_{478}, x_{\asp}z_{478}, x_{9}z_{478}, x_{6}z_{478},\\ & x_{5}z_{478}, x_{3}z_{478}, x_{2}z_{478}, x_{1}z_{478}, x_{\bsp}z_{129}, x_{\asp}z_{129}, x_{8}z_{129}, x_{7}z_{129}, x_{6}z_{129}, x_{5}z_{129}, x_{4}z_{129}, x_{3}z_{129},\\ & x_{\bsp}z_{189}, x_{\asp}z_{189}, x_{7}z_{189}, x_{6}z_{189}, x_{5}z_{189}, x_{4}z_{189}, x_{3}z_{189}, x_{2}z_{189}, x_{\bsp}z_{23\asp}, x_{9}z_{23\asp}, x_{8}z_{23\asp}, x_{7}z_{23\asp},\\ & x_{6}z_{23\asp}, x_{5}z_{23\asp}, x_{4}z_{23\asp}, x_{1}z_{23\asp}, x_{\bsp}z_{34\asp}, x_{9}z_{34\asp}, x_{8}z_{34\asp}, x_{7}z_{34\asp}, x_{6}z_{34\asp}, x_{5}z_{34\asp}, x_{2}z_{34\asp}, x_{1}z_{34\asp},\\ &x_{\bsp}z_{45\asp}, x_{9}z_{45\asp}, x_{8}z_{45\asp}, x_{7}z_{45\asp}, x_{6}z_{45\asp}, x_{3}z_{45\asp}, x_{2}z_{45\asp}, x_{1}z_{45\asp}, x_{\bsp}z_{29\asp}, x_{8}z_{29\asp}, x_{7}z_{29\asp}, x_{6}z_{29\asp},\\
	        &x_{5}z_{29\asp}, x_{4}z_{29\asp}, x_{3}z_{29\asp}, x_{1}z_{29\asp}, x_{\asp}z_{56\bsp}, x_{9}z_{56\bsp}, x_{8}z_{56\bsp}, x_{7}z_{56\bsp}, x_{4}z_{56\bsp}, x_{3}z_{56\bsp}, x_{2}z_{56\bsp}, x_{1}z_{56\bsp},\\
	        &x_{\asp}z_{67\bsp}, x_{9}z_{67\bsp}, x_{8}z_{67\bsp}, x_{5}z_{67\bsp}, x_{4}z_{67\bsp}, x_{3}z_{67\bsp}, x_{2}z_{67\bsp}, x_{1}z_{67\bsp}, x_{\asp}z_{78\bsp}, x_{9}z_{78\bsp}, x_{6}z_{78\bsp}, x_{5}z_{78\bsp},\\ &x_{4}z_{78\bsp}, x_{3}z_{78\bsp}, x_{2}z_{78\bsp}, x_{1}z_{78\bsp}, x_{\asp}z_{89\bsp}, x_{7}z_{89\bsp}, x_{6}z_{89\bsp}, x_{5}z_{89\bsp}, x_{4}z_{89\bsp}, x_{3}z_{89\bsp}, x_{2}z_{89\bsp}, x_{1}z_{89\bsp},\\  &x_{9}z_{5\asp\bsp}, x_{8}z_{5\asp\bsp}, x_{7}z_{5\asp\bsp}, x_{6}z_{5\asp\bsp}, x_{4}z_{5\asp\bsp}, x_{3}z_{5\asp\bsp}, x_{2}z_{5\asp\bsp}, x_{1}z_{5\asp\bsp}, x_{8}z_{9\asp\bsp}, x_{7}z_{9\asp\bsp}, x_{6}z_{9\asp\bsp}, x_{5}z_{9\asp\bsp},\\ &x_{4}z_{9\asp\bsp}, x_{3}z_{9\asp\bsp}, x_{2}z_{9\asp\bsp}, x_{1}z_{9\asp\bsp})\\
	     + \ (&y_{4}z_{145} - y_{6}z_{156}, y_{2}z_{126} - y_{5}z_{156}, y_{2}z_{237} - y_{4}z_{347}, y_{1}z_{126} - y_{7}z_{267}, y_{3}z_{237} - y_{6}z_{267}, y_{5}z_{145} - y_{8}z_{148},\\ & y_{3}z_{347} - y_{8}z_{478}, y_{1}z_{148} - y_{7}z_{478}, y_{6}z_{126} - y_{9}z_{129}, y_{4}z_{148} - y_{9}z_{189}, y_{2}z_{129} - y_{8}z_{189}, y_{\asp}z_{23\asp} - y_{7}z_{237},\\ & y_{\asp}z_{34\asp} - y_{7}z_{347}, y_{2}z_{23\asp} - y_{4}z_{34\asp}, y_{1}z_{145} - y_{\asp}z_{45\asp}, y_{3}z_{34\asp} - y_{5}z_{45\asp}, y_{1}z_{129} - y_{\asp}z_{29\asp}, y_{3}z_{23\asp} - y_{9}z_{29\asp},\\ &y_{1}z_{156} - y_{\bsp}z_{56\bsp}, y_{2}z_{267} - y_{\bsp}z_{67\bsp}, y_{5}z_{56\bsp} - y_{7}z_{67\bsp}, y_{4}z_{45\asp} - y_{\bsp}z_{5\asp\bsp}, y_{\asp}z_{5\asp\bsp} - y_{6}z_{56\bsp}, y_{4}z_{478} - y_{\bsp}z_{78\bsp},\\ &y_{6}z_{67\bsp} - y_{8}z_{78\bsp}, y_{1}z_{189} - y_{\bsp}z_{89\bsp}, y_{7}z_{78\bsp} - y_{9}z_{89\bsp}, y_{2}z_{29\asp} - y_{\bsp}z_{9\asp\bsp}, y_{8}z_{89\bsp} - y_{\asp}z_{9\asp\bsp}, y_{5}z_{5\asp\bsp} - y_{9}z_{9\asp\bsp})\\
	     + \ (&x_{2}x_{8}, x_{4}x_{6}, x_{3}x_{9}, x_{3}x_{6}, x_{7}x_{9}, x_{3}x_{8}, x_{6}x_{8}, x_{b}x_{2}, x_{4}x_{9}, x_{1}x_{a}, x_{5}x_{9}, x_{5}x_{7}, x_{6}x_{9}, x_{1}x_{7}, x_{a}x_{7},\\
	     & x_{1}x_{b}, x_{a}x_{6}, x_{2}x_{5}, x_{a}x_{8}, x_{2}x_{4}, x_{b}x_{4}, x_{3}x_{5}, x_{5}x_{8}, x_{b}x_{3}, x_{1}x_{3})\\
         + \ (&x_{1}y_{1},x_{2}y_{2},x_{3}y_{3},x_{4}y_{4},x_{5}y_{5},x_{6}y_{6},x_{7}y_{7},x_{8}y_{8},x_{9}y_{9},x_{a}y_{a},x_{b}y_{b})\\
         + \  (&z_{145},z_{126},z_{156},z_{237},z_{347},z_{267},z_{148},z_{478},z_{129},z_{189},z_{23a},z_{34a},z_{45a},z_{29a},z_{56b},z_{67b},z_{78b},\\
         &z_{89b},z_{5ab},z_{9ab})^2. 
	    \end{align*}
	    \endgroup
	    
	    The Hilbert series of the Gorenstein algebra $R_{\Delta}$ is equal to
	    \[
	    \frac{1+31t+60t^2+31t^3+t^4}{(1-t)^{11}}.
	    \]
	    We can verify \Cref{cor: koszul_iff_cm} on this example. As any triangulation of the projective plane is Cohen--Macaulay over $\mathbb{F}$ if and only if the characteristic of the field is not equal to $2$, we expect to observe a different behaviour for the resolution of $\mathbb{F}$ as an $R_{\Delta}$-module in the cases $\mathbb{F}=\mathbb{Z}_2$ and $\mathbb{F}=\mathbb{Z}_3$. Indeed, when $\mathbb{F}=\mathbb{Z}_2$, we find a nonlinear syzygy in homological degree $3$, as shown in \Cref{tab: betti}. The numbers in \Cref{tab: betti} have been computed via Macaulay2 \cite{M2}, using the commands {\texttt{ use(T/J); betti res(ideal gens(T/J), LengthLimit $=>$ 3)}}. 
	    \begin{table}[!htb]
	        \centering
	        \begin{minipage}{0.35\textwidth}
	        $\begin{array}{r|c c c c}
	             &  0 & 1 & 2 & 3 \\\hline
	           0  & 1 & 42 & 1297 & 37883\\
	           1 & \text{.} & \text{.} & \text{.} & \mathbf{1}
	        \end{array}$
	        \end{minipage}
	       \begin{minipage}{0.35\textwidth}
	       $\begin{array}{r|c c c c}
	             &  0 & 1 & 2 & 3 \\\hline
	           0  & 1 & 42 & 1297 & 37883\\
	           1 & \text{.} & \text{.} & \text{.} & \mathbf{.}
	        \end{array}$
	        \end{minipage}
	        
	        \vspace{10pt}
	        
	        \caption{The beginning of the Betti tables of $\mathbb{F}$ as an $R_{\Delta}$-module in the cases $\mathbb{F}=\mathbb{Z}_2$ (left) and $\mathbb{F}=\mathbb{Z}_3$ (right). Here $\Delta$ is the flag triangulation of $\mathbb{RP}^2$ from \Cref{ex: RP2}.}
	        \label{tab: betti}
	    \end{table}
	    
	\end{exam}

	\settocdepth{section}
	
	\section{Quadratic Gr\"{o}bner bases and shelling orders}\label{sec: shellable}
	In this section we add a further connection between the combinatorial features of $\Delta$ and the algebraic properties of $R_{\Delta}$. We begin by recalling the notion of shellability.
	\begin{defi} \label{def: shelling}
	    A pure $(d-1)$-dimensional simplicial complex $\Delta$ is \emph{shellable} if there exists a linear order $F_1,\dots,F_{|\mathcal{F}(\Delta)|}$ of its facets such that $\Delta_{i-1}\cap\langle F_i\rangle$ is pure and $(d-2)$-dimensional for every $i =  2,\dots,|\mathcal{F}(\Delta)|$, where $\Delta_{i}=\langle F_1,\dots,F_i \rangle$. The ordering $F_1,\dots,F_{|\mathcal{F}(\Delta)|}$ is called a \emph{shelling order}, and we will often refer to the step $\Delta_{i-1}\to\Delta_i$ as a \emph{shelling}.
	\end{defi}
	This property poses severe restrictions on the topology of $\Delta$, as shellable simplicial complexes are homotopy equivalent to a wedge of some (possibly zero) $(d-1)$-spheres. Boundary complexes of simplicial polytopes of any dimension are shellable, while there exist non-shellable simplicial spheres and balls already in dimension $3$. Moreover, all links of a shellable simplicial complex are again shellable. Shellability has also implications on a more algebraic level, like the following.
	\begin{prop}
	    Shellable simplicial complexes are Cohen--Macaulay over any field.
	\end{prop}
	The converse is far from being true and fails already in dimension $2$: it is well known that any triangulation of the so-called \emph{dunce hat} is Cohen--Macaulay over any field, but not shellable. However, the set of $f$-vectors of Cohen--Macaulay complexes and that of shellable ones coincide in any dimension \cite[Theorem 6]{StaCM}. \\
	The goal of the previous sections was to show a connection between the Cohen--Macaulayness of a flag simplicial complex $\Delta$ and the Koszulness of $R_{\Delta}$. The next result shows that we can tighten both conditions and still obtain a correspondence.
	\begin{theorem}\label{thm: shellable iff qgb}
	    Let $\Delta$ be a pure flag simplicial complex. Then $\Delta$ is shellable if and only if $R_{\Delta}$ has a quadratic Gr\"{o}bner basis.
	\end{theorem}
	Recall that if a standard graded algebra has quadratic Gr\"{o}bner basis, then it is Koszul. Before embarking on the proof of \Cref{thm: shellable iff qgb} we observe the following fact.
	\begin{lem}\label{lem: term order induce alla luce}
	Any term order $<$ on $\F[\mathbf{x},\mathbf{y},\mathbf{z}]$ induces a total order $\prec$ on the facets of $\Delta$, defined as the reflexive closure of the following rule: \[
	    F_{\ell} \prec F_{k} \Leftrightarrow \lt_<(b_{F_k, F_{\ell}}) = \mathbf{y}^{F_k \setminus F_{\ell}}z_{F_k}.
	    \]
	\end{lem}
		\begin{proof}
	    By definition, for any two distinct facets $F_k$ and $F_{\ell}$ of $\Delta$ at most one of $F_{\ell}\prec F_{k}$ and $F_{k}\prec F_{\ell}$ holds. To conclude that $\prec$ is a partial order, we need to prove transitivity.
	    
	    Let $F_j$, $F_k$, $F_{\ell}$ be facets such that $F_{j} \succ F_k$, $F_k \succ F_{\ell}$. We claim that $\lt_<(b_{F_{j}, F_{\ell}}) = \mathbf{y}^{F_{j} \setminus F_{\ell}}z_{F_{j}}$.
	    
	    Indeed, multiplying $b_{F_j, F_k}$ by $\mathbf{y}^{F_j \cap F_k}$ yields that $\mathbf{y}^{F_j}z_{F_j} > \mathbf{y}^{F_k}z_{F_k}$; analogously, one has that $\mathbf{y}^{F_k}z_{F_k} > \mathbf{y}^{F_{\ell}}z_{F_{\ell}}$, and hence $\mathbf{y}^{F_j}z_{F_j} > \mathbf{y}^{F_k}z_{F_k} > \mathbf{y}^{F_{\ell}}z_{F_{\ell}}$. Since $\mathbf{y}^{F_j}z_{F_j} - \mathbf{y}^{F_{\ell}}z_{F_{\ell}} = \mathbf{y}^{F_j \cap F_{\ell}}b_{F_j, F_{\ell}}$, the claim follows.
	\end{proof}
	
	\begin{defi}
	    If $<$ is a term order on $\F[\mathbf{x},\mathbf{y},\mathbf{z}]$, we will denote by $\prec$ be the unique total order on $\mathcal{F}(\Delta)$ induced by $<$ as in \Cref{lem: term order induce alla luce}.  Conversely, if $\prec$ is any total order on $\mathcal{F}(\Delta)$, we will say that a term order $<$ on $\F[\mathbf{x},\mathbf{y},\mathbf{z}]$ is \emph{compatible} with $\prec$ if $\lt_<(b_{F_k, F_{\ell}}) = \mathbf{y}^{F_k \setminus F_{\ell}}z_{F_k}$ whenever $F_{\ell} \prec F_k$.
	\end{defi}

\begin{lem} \label{lem:single_reduction}
Fix a term order $<$ on $\F[\mathbf{x},\mathbf{y},\mathbf{z}]$, let $\prec$ be the associated total order on $\mathcal{F}(\Delta)$ and let $F_1$ and $F_2$ be facets of $\Delta$ with $F_1 \succ F_2$. Then the binomial $b_{F_1, F_2}$ can be reduced via the binomial $b_{G_1, G_2}$ if and only if $\{G_1, G_2\} = \{F_1, H\}$ for some facet $H$ such that $H \prec F_1$ and $F_1 \cap F_2 \subseteq H$. If this is the case, the binomial $b_{F_1, F_2}$ reduces to $\mathbf{y}^{F_2 \cap (H \setminus F_1)}b_{H, F_2}$.    
\end{lem}
\begin{proof}
Since $F_1 \succ F_2$, the leading term of $b_{F_1, F_2}$ is $\mathbf{y}^{F_1 \setminus F_2}z_{F_1}$. We can operate a reduction via $b_{G_1, G_2}$ precisely when the leading term of $b_{G_1, G_2}$ divides $\mathbf{y}^{F_1 \setminus F_2}z_{F_1}$. In order for this to happen, $b_{G_1, G_2}$ must be (up to sign) $b_{F_1, H}$, where $H$ is a facet of $\Delta$. For $z_{F_1}$ to appear in the leading term of $b_{F_1, H}$, it must be that $F_1 \succ H$. Moreover, we need $\mathbf{y}^{F_1 \setminus H}$ to divide $\mathbf{y}^{F_1 \setminus F_2}$, i.e.~$F_1 \setminus H \subseteq F_1 \setminus F_2$, which happens precisely when $F_1 \cap F_2 \subseteq H$ (note that this in turn equivalent to $F_2 \setminus H \subseteq F_2 \setminus F_1$). One checks that all these necessary conditions are indeed sufficient to get the desired reduction. When such conditions are met, one has that
\[
\begin{split}
b_{F_1, F_2} &= \mathbf{y}^{F_1 \setminus F_2}z_{F_1} - \mathbf{y}^{F_2 \setminus F_1}z_{F_2}\\
&= \mathbf{y}^{F_1 \cap (H \setminus F_2)}\mathbf{y}^{F_1 \setminus H}z_{F_1} - \mathbf{y}^{F_2 \cap (H \setminus F_1)}\mathbf{y}^{F_2 \setminus H}z_{F_2}\\
&= \mathbf{y}^{F_1 \cap (H \setminus F_2)}(\mathbf{y}^{F_1 \setminus H}z_{F_1} - \mathbf{y}^{H \setminus F_1}z_{H}) + \mathbf{y}^{F_1 \cap (H \setminus F_2)}\mathbf{y}^{H \setminus F_1}z_{H} - \mathbf{y}^{F_2 \cap (H \setminus F_1)}\mathbf{y}^{F_2 \setminus H}z_{F_2}\\
&= \mathbf{y}^{F_1 \cap (H \setminus F_2)}b_{F_1, H} + \mathbf{y}^{F_2 \cap (H \setminus F_1)}\mathbf{y}^{H \setminus F_2}z_{H} - \mathbf{y}^{F_2 \cap (H \setminus F_1)}\mathbf{y}^{F_2 \setminus H}z_{F_2}\\
&= \mathbf{y}^{F_1 \cap (H \setminus F_2)}b_{F_1, H} + \mathbf{y}^{F_2 \cap (H \setminus F_1)}b_{H, F_2}.
\end{split}
\]
\end{proof}

For the rest of the section we let
\[
    \mathcal{Q}:=\{b_{F_i,F_j}~:~ \dim(F_i\cap F_j)=d-2\}.
\]
Observe that the condition $\dim(F_i\cap F_j)=d-2$ implies that all binomials in $\mathcal{Q}$ are quadratic. Conversely, for every binomial $b_{F_i,F_j}$ with $\deg(b_{F_i,F_j})=2$ we have that $\dim(F_i\cap F_j)=d-2$.

We now establish a technical lemma that will be crucial in the proof of \Cref{thm: shellable iff qgb}.

	\begin{lem}\label{lem: reduce S}
        Fix a term order $<$ on $\F[\mathbf{x},\mathbf{y},\mathbf{z}]$ and let $\prec$ be the associated total order on $\mathcal{F}(\Delta)$. Let $b_{F,F_1},b_{F,F_2}\in \mathcal{Q}$ be such that $\lt(b_{F,F_1})=y_{F\setminus F_1}z_{F}$ and $\lt(b_{F,F_2})=y_{F\setminus F_2}z_{F}$. Then $S(b_{F,F_1},b_{F,F_2})$ reduces to zero modulo $\mathcal{Q}$ if and only if there exists a sequence of facets $F_1=F^{(1)},\dots,F^{(p)}=F_2$ such that
        \begin{enumerate}
            \item $F_1\cap F_2\subset F^{(i)}$ for every $1\leq i \leq p$;
            \item there exists $1\leq c\leq p$ such that
            \[
                F^{(1)}\succ\dots\succ F^{(c)}\prec F^{(c+1)}\prec\dots \prec F^{(p)}.
            \]
            with $\dim(F^{(i)}\cap F^{(i+1)})=d-2$, for every $1\leq i \leq p-1$.\\
        \end{enumerate}
    \end{lem}
    
    \begin{figure}
        \centering
        \includegraphics{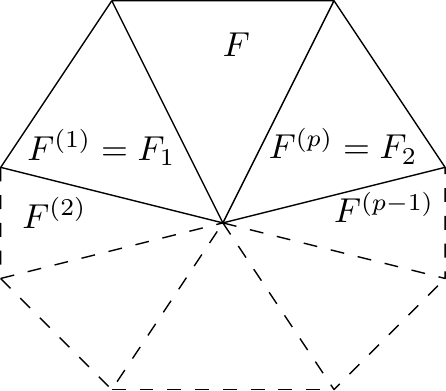}
        \caption{An illustration of \Cref{lem: reduce S}.}
        \label{fig:my_label}
    \end{figure}
    
    \begin{proof}
   The facets $F_1$ and $F_2$ must meet in codimension one or two, and the proof of \Cref{prop:uni gb} yields that $S(b_{F,F_1},b_{F,F_2}) = \mathbf{y}^{(F_1 \cap F_2) \setminus F}b_{F_2,F_1}$.
        
        If $F_1$ and $F_2$ meet in codimension one, then $b_{F_2,F_1}\in\mathcal{Q}$. In particular, $S(b_{F,F_1},b_{F,F_2})$ reduces to zero modulo $\mathcal{Q}$, and the sequence of facets consisting of $F_1$ and $F_2$ always satisfies the hypotheses.
        
        From now on we will hence assume $F_1$ and $F_2$ meet in codimension two. Then $F_1 \cap F_2 \subseteq F$ and thus $S(b_{F,F_1},b_{F,F_2}) = b_{F_2,F_1}$. Assume without loss of generality that $F_1 \succ F_2$.
        
        \begin{itemize}
        \item[\textbf{Only if}:] We assume that $S(b_{F,F_1},b_{F,F_2}) = b_{F_2, F_1}$ reduces to zero modulo $\mathcal{Q}$ and prove the existence of a sequence of facets satisfying the hypotheses.
        By hypothesis, $b_{F_2, F_1} \notin \mathcal{Q}$ becomes zero after applying finitely many reductions via binomials in $\mathcal{Q}$. Let $(b_1, \ldots, b_{p-1})$ (with $p \geq 3$) be the shortest possible sequence of binomials used to reduce $b_{F_2, F_1}$ to zero.
        By \Cref{lem:single_reduction}, $b_{F_2, F_1}$ can be reduced via a binomial in $\mathcal{Q}$ if and only if such a binomial is (up to sign) $b_{F_1, H_1}$, where $H_1$ is a facet meeting $F_1$ in codimension one and such that $F_1 \succ H_1$ and $F_1 \cap F_2 \subseteq H_1$. Moreover, such reduction produces a $y$-multiple of $b_{F_2, H_1}$; hence, if $p=3$, then $F_2$ and $H_1$ must meet in codimension one and $(F_1, H_1, F_2)$ is the desired chain.
        If instead $p>3$, then $F_2$ and $H_1$ must meet in codimension two (since $F_1 \cap F_2 \subseteq H_1 \cap F_2)$. Applying \Cref{lem:single_reduction} to (the given $y$-multiple of) $b_{F_2, H_1}$, we find another facet $H_2$ meeting $\max\{H_1, F_2\}$ in codimension one and such that $\max\{H_1, F_2\} \succ H_2$ and $H_1 \cap F_2 \subseteq H_2$. Moreover, one also has that $F_1 \cap F_2 \subseteq H_1 \cap F_2 \subseteq H_2$, as requested.
        The reduction step produces a $y$-multiple of $b_{\min\{H_1, F_2\}, H_2}$; reasoning as before, if $p=4$ then $\min\{H_1, F_2\}$ and $H_2$ must meet in codimension one and we are done, otherwise we continue the process until we reach the end of the reducing sequence.
        \item[\textbf{If}:] For the reverse implication, assume that we are given a sequence of facets $F^{(1)}\succ\dots\succ F^{(c)}\prec F^{(c+1)}\prec\dots \prec F^{(p)}$ satisfying the hypotheses. Then one gets the desired sequence of binomials by the following algorithm:

\begingroup        
\allowdisplaybreaks

\vspace{10pt}

\noindent $i\gets 1$, $j \gets 1$, $k \gets p$, $b_1, \ldots, b_{p-1} \gets 0$\hfill\\ 
\textbf{while} $i<p-1$ \textbf{and} $\dim(F^{(j)} \cap F^{(k)}) = d-3$ \hfill\\
    \text{\hspace{30pt}}\textbf{if} $\max\{F^{(j)}, F^{(k)}\} = F^{(j)}$\hfill\\
    \text{\hspace{30pt}}\text{\hspace{30pt}}$b_i \gets b_{F^{(j)}, F^{(j+1)}}$, $j\gets j+1$, $i \gets i+1$\hfill\\
    \text{\hspace{30pt}}\textbf{else}\hfill\\
    \text{\hspace{30pt}}\text{\hspace{30pt}}$b_i \gets b_{F^{(k-1)}, F^{(k)}}$, $k\gets k-1$, $i \gets i+1$\hfill\\
    \text{\hspace{30pt}}\textbf{end if}\\
\textbf{end while}\\
$b_i \gets b_{F^{(j)}, F^{(k)}}$\hfill\\
\textbf{return} $b_1, \ldots, b_i$\hfill

\vspace{10pt}
\endgroup

\noindent Note that we are invoking \Cref{lem:single_reduction} at every iteration inside the while cycle. This can be done because, whenever $F^{(j)}$ and $F^{(k)}$ meet in codimension two, then $F^{(j)} \cap F^{(k)} = F_1 \cap F_2$, and $F_1 \cap F_2$ is contained both in $F^{(j+1)}$ and $F^{(k-1)}$ by hypothesis. Moreover, when the condition $\dim(F^{(j)} \cap F^{(k)}) = d-3$ is not met, then it must be that $F^{(j)}$ and $F^{(k)}$ meet in codimension one, and thus we can stop the process after a final reduction via $b_{F^{(j)}, F^{(k)}}$.   
\end{itemize}
        \end{proof} 
    
    We are now ready to prove the main result in this section.
    
	\emph{Proof of \Cref{thm: shellable iff qgb}}
	    We prove the two implications separately.
	    \begin{quotation}\textbf{Only if:} Let $F_1\prec\dots\prec F_{|\mathcal{F}(\Delta)|}$ be a shelling order for $\Delta$. Then the collection
	    \begin{align*}
	        C_{\Delta}~=~&\mathcal{Q}\cup \{x_iy_i : i=1,\dots, n\}\cup \{\mathbf{x}^{N_i} : N_i \text{ minimal nonface of }\Delta\} \cup\\
	        &\{z_{F_i}z_{F_j} : F_i,F_j \text{ facets of }\Delta\}\cup\{x_i z_{F_j}: F_j \text{ facet of }\Delta, \ i\notin F_j\}
	    \end{align*}
	    is a Gr\"{o}bner basis for the ideal defining $R_{\Delta}$ with respect to any term order compatible with $\prec$.
	    \end{quotation}
	    Since $\Delta$ is flag, all monomials in the collection above have degree $2$. Moreover, shellable complexes are $(S_2)$, and hence the algebra $R_{\Delta}$ is quadratic by \Cref{prop:superlevel}. In particular, the defining ideal of $R_{\Delta}$ is generated by $C_{\Delta}$, as this is the set of degree $2$ elements of the generating set described in \Cref{prop:RDelta_presentation}. It is left to the reader to check that all the $S$-polynomials obtained by comparing a monomial and a binomial in $C_{\Delta}$ reduce to zero. Hence, to prove the claim we only need to show that the $S$-polynomials of the form $S(b_1,b_2)$, with $b_1,b_2\in\mathcal{Q}$, reduce to zero. Observe that with the term order we fixed, the leading term of each binomial is determined by the $z$-variables alone. 
	    We begin by noting that if two binomials $b_1$ and $b_2$ are such that $\lt(b_1)=y_iz_F$ and $\lt(b_2)=y_jz_G$, with $F\neq G$, then $S(b_1,b_2)$ reduces to zero, even if $i=j$. This follows from the fact that both monomials of $S(b_1,b_2)$ contain a product of two $z$-variables. Since the collection $C_{\Delta}$ contains all monomials of degree $2$ in the $z$-variables, these $S$-polynomials reduce to zero.
	    
	    Consider now $S(b_{F_k,F_{\ell}},b_{F_k,F_m})$, with $k>\ell$ and $k>m$. By \Cref{lem: reduce S}, the polynomial $S(b_{F_k,F_{\ell}},b_{F_k,F_m})= \mathbf{y}^{(F_{\ell} \cap F_m) \setminus F_k}b_{F_m, F_{\ell}}$ reduces to zero modulo the collection $C_{\Delta}$ if and only if it reduces to zero modulo the collection $C_{\Delta_k}$, with $\Delta_{k}=\langle F_1,\dots,F_{k}\rangle$. 
	    
	    Let $H = F_{\ell}\cap F_{m}\in\Delta_{k-1}$ and note that $|H|\geq d-2$.
	    
	    If $|H|=d-1$, then $b_{F_{\ell},F_m}\in\mathcal{Q}$, which implies that $S(b_{F_k,F_{\ell}},b_{F_k,F_{m}})$ is a multiple of a binomial in $\mathcal{Q}$. In particular, $S(b_{F_k,F_{\ell}},b_{F_k,F_{m}})$ reduces to zero.
	    
	    Assume then that $|H|=d-2$. Since $\Delta_{k-1}$ is shellable, the $1$-dimensional simplicial complex $\lk_{\Delta_{k-1}}(H)$ is shellable, with a shelling order induced by the one of $\Delta_{k-1}$. Since shellable simplicial complexes of dimension greater than zero are connected, there exists a path in $\lk_{\Delta_{k-1}}(H)$ connecting the two edges $F_{\ell}\setminus H$ and $F_{m}\setminus H$. The edges of this path correspond to facets of $\Delta_{k-1}$ via taking their union with $H$: therefore, we can identify the path with a sequence of facets of $\Delta_{k-1}$ which all contain $H$. We can describe the support of such a path via a binary string of length $|\mathcal{F}(\Delta_{k-1})|$ whose $i$-th character is $1$ if $F_i\in P$ and $0$ if $F_i\notin P$. We consider the colexicographic order on these strings, i.e.,
	    \begin{align*}	(s_1,\dots,s_{|\mathcal{F}(\Delta_{k-1})|})<_{\text{colex}} (t_1,\dots,t_{|\mathcal{F}(\Delta_{k-1})|}) \iff s_i<t_i \text{ with } i=\max\{j: s_j\neq t_j\}.
	    \end{align*}
	    
	   It is important to observe that the path between $F_{\ell}$ and $F_{m}$ might not be unique. We choose the one $P\!: F_{\ell} = F_{i_1} \to F_{i_2} \dots \to F_{i_p} = F_m$ whose support is minimal with respect to the colexicographic order on  $2^{\mathcal{F}(\Delta_{k-1})}$.
	   We claim there exists $1\leq c\leq p$ such that $z_{F_{i_1}}>z_{F_{i_2}}>\dots>z_{F_{i_c}}<z_{F_{i_{c+1}}}<\dots<z_{F_{i_p}}$.
	   
	   Assume by contradiction there exist $z_{F_{i_r}}<z_{F_{i_s}}>z_{F_{i_t}}$, with $r<s<t$. The minimality of $P$ implies that one step before the facet $F_{i_s}\setminus H$ of $\lk_{\Delta_{k-1}}(H)$ is added via a shelling, there is no path joining the edges corresponding to $F_{i_r}$ and $F_{i_t}$, which contradicts the connectedness of shellable complexes. Indeed, if such a path existed, we would be able to replace $\{F_{i_{r+1}},\dots,F_{i_s},\ldots, F_{i_{t-1}}\}$ in the support of $P$ with a subset containing only facets which are older than $F_{i_s}$: this would lead to a set colexicographically smaller than the support of $P$, which is impossible. By \Cref{lem: reduce S}, $S(b_{F_k,F_{\ell}},b_{F_k,F_m})$ reduces to zero modulo $C_{\Delta}$.
	    \begin{quotation}
	    \textbf{If:} If $R_{\Delta}$ has a quadratic Gr\"{o}bner basis with respect to $<$, then the total order $\prec$ on $\mathcal{F}(\Delta)$ induced by $<$ is a shelling order for $\Delta$.
	    \end{quotation}
	    Assume by contradiction that the total order induced by $<$ is not a shelling order, and let $F_1 \prec \dots \prec F_{|\mathcal{F}(\Delta)|}$ be the facets in this order. This implies that there exists $2\leq i \leq |\mathcal{F}(\Delta)|$ such that $\Delta_{i-1}\cap \langle F_i\rangle$ is not pure and $(d-2)$-dimensional. In particular, there is a facet $G$ of $\Delta_{i-1}\cap \langle F_i\rangle$ with $\dim(G)\leq d-3$. We observe that $\lk_{\Delta_i}(G)$ is not connected in codimension $1$. Indeed, $G$ is not properly contained in any face of $\Delta_{i-1}\cap \langle F_i\rangle$, so no face in $\lk_{\Delta_i}(G)$ contains vertices of $\Delta_{i-1}\cap \langle F_i\rangle$. In particular, the facet $F_i\setminus G$ does not intersect any other facet of $\lk_{\Delta_{i}}(G)$. However, since $R_{\Delta}$ has a quadratic Gr\"{o}bner basis then it is Koszul, and hence by \Cref{cor: koszul_iff_cm} $ \Delta$ is Cohen--Macaulay. In particular, $\lk_{\Delta}(G)$ is Cohen--Macaulay, and hence it is connected in codimension $1$. This implies that $i<|\mathcal{F}(\Delta)|$. Let $k$ be the smallest integer in $\{i+1,\dots,|\mathcal{F}(\Delta)|\}$ such that $\lk_{\Delta_{k}}(G)$ is connected in codimension 1. Then there exist $1\leq r,s\leq k-1$ such that
	    \begin{itemize}
	        \item $G\subset F_{r}$, $G\subset F_{s}$;
            \item $(F_k\setminus G)\cap (F_r\setminus G)$ and $(F_k\setminus G)\cap (F_s\setminus G)$ are faces of codimension $1$ of $\lk_{\Delta_{k}}(G)$;
            \item $(F_r\setminus G)$ and $(F_s\setminus G)$ are not connected in codimension $1$ in $\lk_{\Delta_{k-1}}(G)$.
        \end{itemize}
        Hence $\dim(F_k\cap F_r)=\dim(F_k\cap F_s)=d-2$, which implies that $b_{F_k,F_r}, b_{F_k,F_s}\in\mathcal{Q}$. 
	    
	    We claim that $S(b_{F_k,F_r}, b_{F_k,F_s})$ does not reduce to zero modulo the set $\mathcal{Q}$. By \Cref{lem: reduce S}, the reduction to zero of this polynomial implies the existence of a sequence of facets $F_r=F^{(1)}\dots,F^{(p)}=F_s$ in $\Delta_{k-1}$ such that
	    \begin{itemize}
	    \item $F^{(t)}\preceq \max\{F_r,F_s\}$ for every $1\leq t\leq p$;
	    \item $G\subset F_r\cap F_s\subset F^{(t)}$ for every $1\leq t\leq p$;
	    \item $\dim(F^{(t)}\cap F^{(t+1)})=d-2$ for every $1\leq t\leq p-1$.
	    \end{itemize}
	    The first condition implies that $F^{(t)}\in\Delta_{k-1}$, the second yields that $F^{(t)}\setminus G$ is a facet of $\lk_{\Delta_{k-1}}(G)$, while the third implies the existence of a path of facets of $\lk_{\Delta_{k-1}}(G)$ connecting $F_r\setminus G$ and $F_s\setminus G$. However, such a sequence does not exist, as $F_r\setminus G$ and $F_s\setminus G$ are not connected in codimension $1$ in $\lk_{\Delta_{k-1}}(G)$. This proves that if $\Delta$ is not shellable, then the set $C_{\Delta}$ is not a Gr\"{o}bner basis. To conclude that $R_{\Delta}$ has no quadratic Gr\"{o}bner basis we use the fact that set $\mathcal{U}$ given in \Cref{prop:RDelta_presentation} is a universal Gr\"{o}bner basis for $R_{\Delta}$ by \Cref{prop:uni gb}, and hence every quadratic reduced Gr\"{o}bner basis must be a subset of the set of quadratic polynomials in  $\mathcal{U}$, which is precisely $C_{\Delta}$.
	\hfill $\square$
    \begin{exam}
        Let $\Delta=\langle123,234,345\rangle$ as in \Cref{ex: section 4}. Up to sign, the only two binomials in $\mathcal{Q}$ are
        \[
            f = y_4z_{234}-y_1z_{123}~ \text{and} ~ g = y_5z_{345}-y_2z_{234}.
        \]
        If we consider a term order compatible with the shelling order  $123\prec 234\prec 345$, then we have $\lt(f)=y_4z_{234}$ and $\lt(g)=y_5z_{345}$. Since the two leading terms are coprime in the $z$-variables, $S(f,g)$ reduces to zero modulo the minimal generating set of $(z_{123},z_{234},z_{345})^2$. On the other hand, if we consider a term order compatible with $123\prec 345\prec 234$, which is not a shelling order, we obtain $\lt(f)=y_4z_{234}$ and $\lt(g)=y_2z_{234}$. In this case we see that none of the monomials in $S(f,g)=y_4y_5z_{345} - y_1y_2z_{123}$ is divisible by any of the leading terms of the elements of $\mathcal{Q}$. Hence, $S(f,g)$ does not reduce to zero.
    \end{exam}

    \section{\texorpdfstring{The $\gamma$-vector of $R_{\Delta}$}{The gamma-vector of RDelta}} \label{sec: gamma}
        The aim of this section is to study the \emph{$\gamma$-vector} of the algebra $R_{\Delta}$ in terms of combinatorial properties of the simplicial complex $\Delta$. In general, the $\gamma$-vector is often the ``right'' invariant to consider when we want to highlight the information encoded in a palindromic vector, like for instance the $h$-vector of a simplicial sphere. See \cite{Bra1}, \cite{Gal}, and \cite{Ath} for more information on this topic.
        
    In particular, recall that the coefficients of the $h$-polynomial of a standard graded Gorenstein algebra form a palindromic sequence. More precisely, if the degree of the $h$-polynomial (which coincides with the Castelnuovo--Mumford regularity since the algebra is Cohen--Macaulay) equals $s$, we have that $h_i=h_{s-i}$ for every $i$. The vector space of univariate palindromic polynomials of degree $s$ has dimension $\lfloor\frac{s}{2}\rfloor+1$, and Gal proposed the use of the basis $\{t^i(t+1)^{s-2i}\}_{i=0}^{\lfloor\frac{s}{2}\rfloor}$ \cite[Section 2]{Gal}.
    \begin{defi}\label{def: gamma vector}
        The \emph{$\gamma$-vector} associated with the integer vector $h=(h_0,\dots,h_s)$ with $h_i=h_{s-i}$ is the integer vector $\gamma=(\gamma_0,\dots,\gamma_{\lfloor\frac{s}{2}\rfloor})$ defined by the identity
        \begin{equation} \label{eq: gamma_def}
        \sum_{i=0}^{\lfloor\frac{s}{2}\rfloor}\gamma_i t^i(t+1)^{s-2i}=\sum_{i=0}^{s}h_it^i.
        \end{equation}
        We denote by $\gamma(A)$ the $\gamma$-vector associated with the $h$-vector of a standard graded Gorenstein algebra $A$.
    \end{defi}
    Each $\gamma_i$ can be expressed as an integer linear combination of the $h_i$: for instance, one has that $\gamma_0=h_0$, $\gamma_1=h_1-sh_0$ and $\gamma_2=h_2-(s-2)h_1+\frac{s(s-3)}{2}h_0$. In general, the following recursive formula holds:
    \begin{equation}\label{eq: recursive}
        \gamma_i = h_i - \sum_{j=0}^{i-1}\binom{s-2j}{i-j}\gamma_j.
    \end{equation}

    For the rest of the section we will focus on the $\gamma$-vector of the standard graded Gorenstein algebra $R_{\Delta}$, with $\Delta$ a pure simplicial complex. If $\Delta$ is $(d-1)$-dimensional, then the $h$-polynomial of $R_{\Delta}$ has degree $d+1$, and we will then replace $s$ with $d+1$ in \Cref{def: gamma vector}.
    
    Perhaps surprisingly, we will show that when $R_{\Delta}$ is Koszul over $\F$ -- i.e., by \Cref{cor: koszul_iff_cm}, when $\Delta$ is Cohen--Macaulay over $\F$ -- many of the $\gamma_i$ are nonpositive.
    
    We begin with a result on the last entry of the $\gamma$-vector. Even though we will prove a more general result in \Cref{cor: gamma alternate}, it is instructive to consider this special case separately, as its proof is rather elementary.
    
	\begin{lem}\label{lem: top gamma}
	    For any $(d-1)$-dimensional Cohen--Macaulay simplicial complex $\Delta$, with $d\geq 3$ odd, we have that
	    \[
	        \gamma_{\frac{d+1}{2}}(R_{\Delta})=(-1)^{\frac{d-1}{2}}2\widetilde{\chi}(\Delta)=(-1)^{\frac{d-1}{2}}2\dim_{\F}\widetilde{H}_{d-1}(\Delta;\F).
	    \]
	     In particular, if $d\equiv 3 \text{ mod }4$, then $\gamma_{\frac{d+1}{2}}(R_{\Delta})$ is nonpositive; moreover, for any integer $c \geq 0$ and every $d\equiv 3 \text{ mod }4$, there exists a Cohen--Macaulay (even shellable) $(d-1)$-complex $\Delta$ such that $\gamma_{\frac{d+1}{2}}(R_{\Delta})=-2c$. Such a complex can be chosen to be flag.
	\end{lem}
	
	\begin{proof}
	    Let $f(\Delta)=(1,f_0,\dots,f_{d-1})$ be the $f$-vector of $\Delta$ and recall that, as noted in \Cref{rem:hvector_RDelta}, $h(R_{\Delta})=(1,f_0+f_{d-1},f_1+f_{d-2},\dots,f_0+f_{d-1},1)$. Evaluating \eqref{eq: gamma_def} at $t=-1$ yields
	    \[
	    (-1)^{\frac{d+1}{2}}\gamma_{\frac{d+1}{2}}(R_{\Delta}) = \sum_{i=0}^{d+1}(-1)^ih_i(R_{\Delta}) = 2 + \sum_{i=1}^{d}(-1)^i(f_{i-1}+f_{d-i}) = -2\widetilde{\chi}(\Delta).
	    \]
	    It is a standard fact from algebraic topology that the reduced Euler characteristic can be written as the alternating sum of the dimensions of the reduced homology groups of $\Delta$. Since $\Delta$ is Cohen--Macaulay, one has that $\widetilde{H}_{j}(\Delta;\F)=0$ for every $0\leq j<d-1$, and thus $\gamma_{\frac{d+1}{2}}(R_{\Delta})=(-1)^{\frac{d-1}{2}}2\dim_{\F}\widetilde{H}_{d-1}(\Delta;\F)$.
	    
	    To prove the last claim it is enough to exhibit a $(d-1)$-dimensional  Cohen--Macaulay flag complex with $\widetilde{\chi}(\Delta)=c$. For instance, we can glue together $c$ copies of the boundary complex of the $d$-dimensional cross-polytope so that they all share the same single facet. It is elementary to check that this simplicial complex has the desired properties, and that it is indeed even shellable.
	\end{proof}
	
	In their collection \cite{PeSt} of problems on syzygies and Hilbert functions, Peeva and Stillman suggest that it might make sense to consider a Charney--Davis-like conjecture for (not necessarily monomial) Koszul Gorenstein algebras:
	\begin{ques}[{\cite[Problem 10.3]{PeSt}}] \label{qu:PeSt}
	    Let $S/I$ be a Koszul Gorenstein algebra with $h$-vector $(h_0, h_1, \ldots, h_{2e})$. Is it true that $(-1)^e(h_0 - h_1 + h_2 - \ldots + h_{2e}) \geq 0$?
	\end{ques}
	Notice that $(-1)^e(h_0 - h_1 + h_2 - \ldots + h_{2e}) = \gamma_{e}(S/I)$. \Cref{lem: top gamma} and \Cref{cor: koszul_iff_cm} immediately yield a negative answer to \Cref{qu:PeSt}:
	\begin{coro}
	    Let $d\equiv 3 \text{ mod }4$, $c \in \mathbb{Z}_+$. Then there exists a Koszul Gorenstein algebra with $h$-vector $(h_0, h_1, \dots, h_{d+1})$ and such that $(-1)^{\frac{d+1}{2}}(h_0 - h_1 + h_2 - \ldots + h_{2e}) = -2c < 0$. Such an algebra is of the form $R_{\Delta}$, where $\Delta$ is a flag Cohen--Macaulay $(d-1)$-dimensional complex with $\dim_{\F}\widetilde{H}_{d-1}(\Delta;\F) = c$.
	\end{coro}
	    
	\begin{exam} \label{ex: algebraic CD counterexample}
	Let $\Delta$ be the boundary of the $3$-dimensional cross-polytope (or octahedron). Since $f(\Delta) = (1, 6, 12, 8)$, one has that $h(R_{\Delta}) = (1, 14, 24, 14, 1)$ and hence $\gamma_2(R_{\Delta}) = 1 - 14 + 24 - 14 + 1 = -2 < 0$. Labeling the antipodal pairs of vertices of $\Delta$ by $\{1,2\}$, $\{3,4\}$ and $\{5,6\}$ yields that $R_{\Delta}$ is presented as the quotient of the polynomial ring in $2f_0(\Delta)+f_2(\Delta)=20$ variables
	\[\F[x_1, x_2, x_3, x_4, x_5, x_6, y_1, y_2, y_3, y_4, y_5, y_6, z_{135}, z_{136}, z_{145}, z_{146}, z_{235}, z_{236}, z_{245}, z_{246}]\] by the ideal with 81 quadratic generators
	\begingroup
	\allowdisplaybreaks
	\begin{align*}
	(&x_1x_2, x_3x_4, x_5x_6, x_1y_1, x_2y_2, x_3y_3, x_4y_4, x_5y_5, x_6y_6)\\
	+\ (&z_{135}, z_{136}, z_{145}, z_{146}, z_{235}, z_{236}, z_{245}, z_{246})^2\\
	+\ (&x_2z_{135}, x_4z_{135}, x_6z_{135}, x_2z_{136}, x_4z_{136}, x_5z_{136}, x_2z_{145}, x_3z_{145}, x_6z_{145}, x_2z_{146}, x_3z_{146}, x_5z_{146},\\
	&x_1z_{235}, x_4z_{235}, x_6z_{235}, x_1z_{236}, x_4z_{236}, x_5z_{236}, x_1z_{245}, x_3z_{245}, x_6z_{245}, x_1z_{246}, x_3z_{246}, x_5z_{246})\\
	+\ (&y_5z_{135} - y_6z_{136}, y_3z_{135} - y_4z_{145}, y_1z_{135} - y_2z_{235}, y_3z_{136} - y_4z_{146}, y_1z_{136} - y_2z_{236}, y_5z_{145} - y_6z_{146},\\ &y_1z_{145} - y_2z_{245}, y_1z_{146} - y_2z_{246}, y_5z_{235} - y_6z_{236}, y_3z_{235} - y_4z_{245}, y_3z_{236} - y_4z_{246}, y_5z_{245} - y_6z_{246}).
	\end{align*}
	\endgroup
	Being the boundary complex of a simplicial polytope, the simplicial complex $\Delta$ is shellable. Therefore, by \Cref{thm: shellable iff qgb} the $81$ generators listed above are a Gr\"{o}bner basis with respect to any term order compatible with any shelling order. We remark that, by Propositions \ref{prop: ADelta from RDelta} and \ref{prop: ADelta qgb}, one can get an analogous Artinian example with 14 variables by going modulo the regular sequence of linear forms $y_1 - x_1, \ldots, y_6 - x_6$. 
	\end{exam}
	
	\begin{rema}
	    \Cref{ex: algebraic CD counterexample} also provides a Koszul Gorenstein algebra which is not PF (in the sense of \cite{ReWe}) and has Castelnuovo--Mumford regularity $4$. By \cite[Corollary 4.14]{ReWe}, all Koszul Gorenstein algebras of regularity at most $3$ are PF and hence have the Charney--Davis (CD) property.
	\end{rema}
	
	\begin{exam}
	    Let $\Delta$ be the flag triangulation of $\mathbb{RP}^2$ described in \Cref{ex: RP2}. We have observed that $h(R_{\Delta})=(1,31,60,31,1)$, and hence $\gamma_2(R_{\Delta})=0$. This is in line with \Cref{lem: top gamma}, as $\dim_{\mathbb{F}}\widetilde{H}_2(\mathbb{RP}^2;\mathbb{F})=0$ for every field $\mathbb{F}$ over which $\Delta$ is Cohen--Macaulay, i.e., every field of characteristic different than $2$.
	\end{exam}
	For the rest of the section we stipulate that a binomial coefficient $\binom{n}{k}$ is equal to zero whenever $k<0$ or $n<k$. Moreover, whenever $r \geq 2i \geq 0$, we set
	
	\[\ell_{r, i} :=
	\begin{cases}
	2 & \text{if }(r,i)=(0,0)\\
	\binom{r-i}{i} + \binom{r-i-1}{i-1} = \frac{r}{r-i}\binom{r-i}{i} & \text{otherwise.}
	\end{cases}
	\]
	
	The positive integers $\ell_{r, i}$ appear as coefficients of \emph{Lucas polynomials}\footnote{There appear to be at least two different families of polynomials known as Lucas polynomials, but in both cases the nonzero coefficients are the integers $\ell_{r, i}$.} in sequence A034807 of the On-Line Encyclopedia of Integer Sequences \cite{OEIS}.
	
	We begin with a technical lemma. 
	\begin{lem}\label{lem: carta di identities}
    The following identities hold for any $n,m,r\geq0$:
	    \begin{enumerate}
	        \item \[ 1+t^r = \sum_{i=0}^{\lfloor\frac{r}{2}\rfloor} (-1)^i\ell_{r,i} t^i (1+t)^{r-2i}\]
	    \item
	    \[
	        \binom{n}{m}+\binom{n}{m-r} = \sum_{i=0}^{\lfloor\frac{r}{2}\rfloor}(-1)^i\ell_{r,i}\binom{n+r-2i}{m-i}.
	    \]
	    \end{enumerate}
	\end{lem}
	\begin{proof}
	Statement $i$ is a specialization of a classical identity due to Girard and Waring \cite[Identity 1]{Gould} which can be derived from the Newton formulas relating power sums and elementary symmetric polynomials.
	
	Expanding $(1+t)^{r-2i}$ inside identity $i$ yields that
	\[1 + t^r = \sum_{i=0}^{\lfloor\frac{r}{2}\rfloor}\sum_{j=0}^{r-2i}(-1)^i\ell_{r,i}\binom{r-2i}{j}t^{i+j}.\]
	This is now an identity inside the free $\mathbb{Z}$-module with basis $\{1, t, t^2, \ldots, t^r\}$. We can now specialize this by substituting each $t^i$ with the binomial coefficient $\binom{n}{m-i}$, obtaining \[\binom{n}{m}+\binom{n}{m-r} = \sum_{i=0}^{\lfloor\frac{r}{2}\rfloor}(-1)^i\ell_{r,i}\left(\sum_{j=0}^{r-2i}\binom{r-2i}{j}\binom{n}{m-i-j}\right).\] Applying Vandermonde's identity for binomial coefficients yields identity $ii$.
	\end{proof}
		
	\begin{prop}\label{prop: formula gamma}
	    Let $\Delta$ be a $(d-1)$-dimensional pure simplicial complex. Then
	    \[
	        \gamma_i(R_{\Delta})=(-1)^{i-1}\sum_{k=2i-1}^d \ell_{k-1, i-1}h_k(\Delta),
	    \] 
	    for every $i\geq 1$. In particular, each $\gamma_i(R_{\Delta})$ is $(-1)^{i-1}$ times a linear combination of the $h_k(\Delta)$ with positive integer coefficients.
	\end{prop}
	
	\begin{proof}
	    Throughout this proof we will set $\gamma_j := \gamma_j(R_{\Delta})$, $h_j := h_j(\Delta)$ and $f_j := f_j(\Delta)$. We prove the desired formula for $\gamma_i$ by induction on $i \geq 1$.
	    
	    For the base case $i=1$ it suffices to observe that $\gamma_1=h_{1}(R_{\Delta})-d-1=f_0+f_{d-1}-d-1=2h_1+\sum_{k=2}^d h_k$, where we used \eqref{eq: f from h} to express $f_0$ and $f_{d-1}$ as functions of the $h_j$.
	    
	    Let us now fix $i>1$ and assume that the claim holds for every $1 \leq j < i$. \eqref{eq: recursive} gives us a way to express $\gamma_i$ as a function of the $\gamma_j$ with $j < i$; moreover, using again \eqref{eq: f from h}, we can explicitly write each $h_i(R_{\Delta})=f_{i-1}+f_{d-i}$ as a combination of the $h_i = h_i(\Delta)$. Putting together these two facts with the inductive hypothesis on the $\gamma_j$ with $j < i$ and the observation that $\gamma_0 = h_0 = 1$, we can write
	    \begin{align*}
	        \gamma_i &= h_i(R_{\Delta}) - \sum_{j=0}^{i-1}\binom{d+1-2j}{i-j}\gamma_j\\
	        &=f_{i-1}+f_{d-i}-\binom{d+1}{i}\gamma_0-\sum_{j=1}^{i-1}\binom{d+1-2j}{i-j}\gamma_j\\
	        &=\sum_{k=0}^i \binom{d-k}{d-i}h_k+\sum_{k=0}^{d-i+1}\binom{d-k}{i-1}h_k-\binom{d+1}{i}h_0-\sum_{j=0}^{i-2}\binom{d-1-2j}{i-j-1}\gamma_{j+1}\\
	        &=\sum_{k=0}^i \binom{d-k}{d-i}h_k+\sum_{k=0}^{d-i+1}\binom{d-k}{i-1}h_k-\binom{d+1}{i}h_0-\sum_{j=0}^{i-2}\binom{d-1-2j}{i-j-1}(-1)^{j}\sum_{s=2j+1}^{d}\ell_{s-1, j}h_s.
	    \end{align*}

	   We have thus obtained an expression for $\gamma_i$ as a linear combination of the $h_j$. We will denote by $[\gamma_i]_k$ the coefficient of $h_k$ in this expression and will analyze such coefficients separately.
	   
	   One has immediately that $[\gamma_i]_0 = \binom{d}{i}+\binom{d}{i-1}-\binom{d+1}{i}=0$. When $0 < k \leq d$, we obtain
	   
	   \begin{equation} \label{eq:kth_coefficient}
	        [\gamma_i]_k=\binom{d-k}{d-i}+\binom{d-k}{i-1}-\sum_{j=0}^{\min\{i-2, \lfloor\frac{k-1}{2}\rfloor\}}\binom{d-1-2j}{i-j-1}(-1)^{j}\ell_{k-1, j}.
	   \end{equation}
	   
	   We wish to rewrite $\binom{d-k}{d-i}+\binom{d-k}{i-1}$ in the above expression. Noting that $\binom{d-k}{d-i} = \binom{d-k}{i-1-(k-1)}$, we are in the position to apply \Cref{lem: carta di identities}.ii with $n=d-k$, $m=i-1$ and $r=k-1$, obtaining that \[\binom{d-k}{d-i}+\binom{d-k}{i-1}=\sum_{j=0}^{\lfloor\frac{k-1}{2}\rfloor}\binom{d-1-2j}{i-j-1}(-1)^j \ell_{k-1, j}.\]
	   
	   Hence, \eqref{eq:kth_coefficient} immediately yields that $[\gamma_i]_k = 0$ whenever $\min\{i-2, \lfloor\frac{k-1}{2}\rfloor\} = \lfloor\frac{k-1}{2}\rfloor,$ which happens when $k \leq 2i-2$.
	   
	   Now assume that $2i-1 \leq k \leq d$. Then $\min\{i-2, \lfloor\frac{k-1}{2}\rfloor\} = i-2$, and so  \eqref{eq:kth_coefficient} becomes
	   \[[\gamma_i]_k = \sum_{j=i-1}^{\lfloor\frac{k-1}{2}\rfloor}\binom{d-1-2j}{i-j-1}(-1)^{j}\ell_{k-1, j} = (-1)^{i-1}\ell_{k-1, i-1},\]
	   with the last equality holding since the binomial coefficient $\binom{d-1-2j}{i-j-1}$ equals $1$ when $j=i-1$ and vanishes otherwise.
	\end{proof}
  
	\begin{coro}\label{cor: gamma alternate}
	    Let $\Delta$ be a  $(d-1)$-dimensional simplicial complex for which $h_i(\Delta)\geq 0$ holds for every $0\leq i \leq d$. Then 
	    \[
	        (-1)^{i-1}\gamma_i(R_{\Delta})\geq 0,
	    \]     
	    for every $1 \leq i \leq \lfloor\frac{d+1}{2}\rfloor$.
	\end{coro}
	In particular, the algebra $R_{\Delta}$ has 
	$\gamma$-numbers which alternate in sign whenever $\Delta$ is Cohen--Macaulay. However, the hypothesis in \Cref{cor: gamma alternate} is satisfied by larger families of simplicial complexes, such as partitionable complexes.

	\begin{rema}
	    We would like to stress that the formula in \Cref{prop: formula gamma} is of purely combinatorial nature, and it holds for more general integer vectors than $f$- and $h$-vectors of simplicial complexes.  Let $\mathfrak{a}=(a_0,a_1,\dots,a_{d+1})$ be any sequence of integers with $a_i=a_{d+1-i}$, and let $\mathfrak{b}=(b_0,b_1,\dots,b_d)$ be a sequence which satisfies $b_0=a_0$ and $b_i+b_{d+1-i}=a_i$, for every $0<i\leq d$. \Cref{prop: formula gamma} allows to express the $\gamma$-vector of $\mathfrak{a}$ as a linear combination of the entries of a vector $\mathfrak{c}$, obtained from $\mathfrak{b}$ via the linear transformation \eqref{eq: h from f}.
	\end{rema}
	
	\section{\texorpdfstring{An Artinian reduction of $R_{\Delta}$ and a connection to work of Gondim and Zappal\`a}{An Artinian reduction of RDelta and a connection to work of Gondim and Zappal\`a}}
	\label{sec: GZ}
	After the first version of this paper was completed, we noticed a connection between the algebras $R_{\Delta}$ defined here and the Artinian Gorenstein algebras $A_{\Delta}$ introduced by Gondim and Zappal\`a in their paper \cite{GZ}: more specifically, $A_{\Delta}$ is an Artinian reduction of $R_{\Delta}$ in characteristic zero. The aim of this final section is to clarify this connection.
	
	For this section, let the characteristic of the field $\mathbb{F}$ be zero\footnote{This is only needed in order to use Macaulay's inverse system, but does not really play a role in our observations: see Remark \ref{rem: char zero}.}. Under this assumption, it is known that every graded Artinian Gorenstein algebra corresponds to a single homogeneous polynomial $f$ via Macaulay's inverse system (see e.g.~\cite[Section 0.2]{IK}). The key idea of the paper by Gondim and Zappal\`a is the following: given a pure simplicial complex $\Delta$ with $n$ vertices, consider the polynomial ring $\mathbb{F}[x_i, z_F \mid i \in \{1, \ldots, n\},\  F \text{ facet of }\Delta]$ and let \[f_{\Delta} := \sum_{F \text{ facet of }\Delta}z_F \cdot \prod_{i \in F}x_i.\]
	
	Call $A_{\Delta}$ the Artinian Gorenstein algebra corresponding to $f_{\Delta}$ via Macaulay's inverse system. After unifying the notation (as the $U$ and $X$-variables in \cite{GZ} are respectively our $x$- and $z$-variables), an analysis of the presentation of $A_{\Delta}$ exhibited in \cite[Theorem 3.2.5]{GZ} reveals that $A_{\Delta} = R_{\Delta} / (y_i - x_i \mid i \in \{1, \ldots, n\})$. Our first observation is then that $y_1 - x_1, \ldots, y_n - x_n$ is a regular sequence of linear forms:
	
	\begin{prop} \label{prop: ADelta from RDelta}
	Let $\Delta$ be a pure simplicial complex and let $\mathrm{char}(\mathbb{F}) = 0$. Then $A_{\Delta} = R_{\Delta} / (y_i - x_i \mid i \in \{1, \ldots, n\})$, and $y_1 - x_1, \ldots, y_n - x_n$ is an $R_{\Delta}$-regular sequence of linear forms.
	\end{prop}
	\begin{proof}
	Since $R_{\Delta}$ is Cohen--Macaulay, it is enough to prove that $\dim R_{\Delta} = n + \dim R_{\Delta} / (y_i - x_i \mid i \in \{1, \ldots, n\})$, see e.g.~the graded version of \cite[Theorem 2.1.2.c]{BH}. Since $R_{\Delta} / (y_i - x_i \mid i \in \{1, \ldots, n\})$ is Artinian and the Krull dimension of $R_{\Delta}$ equals $n$ (check Notation \ref{notat:main}), the claim follows.
	\end{proof}
	
	\begin{coro} \label{cor: ADelta Koszul}
	Let $\Delta$ be a pure flag simplicial complex and let $\mathrm{char}(\mathbb{F}) = 0$. Then the following conditions are equivalent: 
	\begin{enumerate}
	\item $A_{\Delta}$ is Koszul;
	\item $R_{\Delta}$ is Koszul;
	\item $\Delta$ is Cohen--Macaulay over $\mathbb{F}$.
	\end{enumerate}
	\end{coro}
	\begin{proof}
	Conditions ii and iii are equivalent by Corollary \ref{cor: koszul_iff_cm}. If $B$ is a standard graded $\mathbb{F}$-algebra and $\ell \in B_1$ is a regular element, one has that $B$ is Koszul if and only if $B/\ell B$ is \cite[Theorem 4.e.iv]{BF}; hence, the equivalence of conditions i and ii descends directly from Proposition \ref{prop: ADelta from RDelta}.
	\end{proof}
	
	If  $B$ is a standard graded $\mathbb{F}$-algebra, $\ell \in B_1$ is a regular element and $B/\ell{B}$ has a quadratic Gr\"obner basis, then so does $B$ \cite[Lemma 4]{Conca00}; the reverse implication does \emph{not} hold in general. In our case, however, the proof of the ``only if'' implication of Theorem \ref{thm: shellable iff qgb} goes through for $A_{\Delta}$ as well: the only difference lies in the presence of some new nontrivial $S$-pairs, namely the ones coming from a nonface-monomial $\mathbf{x}^N$ and a binomial $\mathbf{x}^{F_1 \setminus F_2}z_{F_1} - \mathbf{x}^{F_2 \setminus F_1}z_{F_2}$. It can be checked that such $S$-pairs reduce to zero. Putting these observations together, we get the following result:
	\begin{prop} \label{prop: ADelta qgb}
		Let $\Delta$ be a pure flag simplicial complex and let $\mathrm{char}(\mathbb{F}) = 0$. Then the following conditions are equivalent: 
	\begin{enumerate}
	\item $A_{\Delta}$ has a quadratic Gr\"obner basis;
	\item $R_{\Delta}$ has a quadratic Gr\"obner basis;
	\item $\Delta$ is shellable.
	\end{enumerate}
	\end{prop}
	
	\begin{rema} \label{rem: char zero}
	The running hypothesis of this section about characteristic zero is needed just to fit the original definition of $A_{\Delta}$ by Gondim and Zappal\`a, where Macaulay's inverse system is used. However, the proof that $y_1 - x_1, \ldots, y_n - x_n$ is an $R_{\Delta}$-regular sequence of linear forms is characteristic-free, and so are the proofs of Corollary \ref{cor: ADelta Koszul} and Proposition \ref{prop: ADelta qgb}, if we substitute ``$A_{\Delta}$'' by ``$R_{\Delta} / (y_i - x_i \mid i \in \{1, \ldots, n\})$'' in the statements. In particular, the results in Section \ref{subsec: Koszulness and characteristic} and Section \ref{sec: gamma} can be adapted to fit the Artinian setting.
	\end{rema}
	
	\begin{rema} \label{rem: quadratic ADelta}
	Because of Proposition \ref{prop: ADelta from RDelta}, the algebra $A_{\Delta}$ has a quadratic defining ideal precisely when $R_{\Delta}$ does; in particular, the characterizations of quadraticity in \cite[Theorem 3.5]{GZ} and in our Proposition \ref{prop:superlevel} should coincide. However, Proposition \ref{prop:superlevel} requires $\Delta$ to be flag and $(S_2)$, while \cite[Theorem 3.5]{GZ} only asks for $\Delta$ to be flag and strongly connected, which is weaker.

We claim that the complex $\Delta$ with facet list $\{123, 235, 245, 457, 567, 167\}$ provides a counterexample to \cite[Theorem 3.5]{GZ}: indeed, $\Delta$ is strongly connected but not $(S_2)$, and the cubical generator $x_2x_3z_{123} - x_6x_7z_{167}$ is needed in the presentation of $A_{\Delta}$, as can be checked for instance with Macaulay2 using the InverseSystems package \cite{ISPackage}.

Note that the examples constructed via Tur{\'a}n complexes in \cite{GZ} are still valid: since Tur{\'a}n complexes are flag and Cohen--Macaulay, by Corollary \ref{cor: ADelta Koszul} the associated algebras are Koszul, and hence quadratically presented.
	\end{rema}

\settocdepth{part}
\section*{Funding}
The first-named author was partially supported by the EPSRC grant EP/R02300X/1. The second-named author is funded by the G\"oran Gustafsson foundation. 

\section*{Acknowledgements}
The first seeds for this paper were planted in June 2019, when the first-named author found counterexamples to the algebraic generalization of the Charney--Davis conjecture mentioned above via some Macaulay2 \cite{M2} experiments. The project then took shape in December 2019, while the first-named author was visiting the second-named author at the Max Planck Institute for Mathematics in the Sciences in Leipzig, Germany. The first-named author is grateful to the institute for the hospitality.

We wish to thank Aldo Conca, Rodrigo Gondim, Vic Reiner, Naoki Terai, Matteo Varbaro, Volkmar Welker and Giuseppe Zappal\`a for useful discussions, and an anonymous referee for their thoughtful comments. Special thanks go to Eran Nevo for an insightful exchange of emails which put us on the right track. Finally, we thank Bruno Vallette for pointing us to further connections between the Koszul and Cohen--Macaulay properties in the literature.

	
\bibliographystyle{alpha}
\bibliography{bibliography}
\end{document}